\documentclass[10pt,reqno]{amsart}
\usepackage{amsmath,multicol}
\usepackage{amsthm}
\usepackage{amssymb}
\usepackage{amsfonts}
\usepackage{subcaption}
\usepackage{graphicx}
\usepackage{latexsym}
\usepackage{cite, url, xcolor}
\usepackage{stmaryrd}
\usepackage{hyperref}
\usepackage{enumitem}
    \setenumerate{itemsep=5pt} % adjust space between list items
    \setitemize{itemsep=5pt} % adjust space between list items

\usepackage{mathtools}
\makeatletter
\DeclareRobustCommand\widecheck[1]{{\mathpalette\@widecheck{#1}}}
\def\@widecheck#1#2{%
    \setbox\z@\hbox{\m@th$#1#2$}%
    \setbox\tw@\hbox{\m@th$#1%
       \widehat{%
          \vrule\@width\z@\@height\ht\z@
          \vrule\@height\z@\@width\wd\z@}$}%
    \dp\tw@-\ht\z@
    \@tempdima\ht\z@ \advance\@tempdima2\ht\tw@ \divide\@tempdima\thr@@
    \setbox\tw@\hbox{%
       \raise\@tempdima\hbox{\scalebox{1}[-1]{\lower\@tempdima\box
\tw@}}}%
    {\ooalign{\box\tw@ \cr \box\z@}}}
\makeatother    
\usepackage[font=Large]{caption}

\usepackage[font=footnotesize,labelfont=bf]{caption}

\usepackage[sc]{mathpazo}
\usepackage[T1]{fontenc}

\usepackage{tikz,amsmath,amsfonts}
\usetikzlibrary{positioning}
\usetikzlibrary{decorations.pathreplacing}

\graphicspath{{IMAGES/}}
\newcommand{\NN}{\mathbb{N}}
\newcommand{\ZZ}{\mathbb{Z}}
\newcommand{\RR}{\mathbb{R}}
\newcommand{\CC}{\mathbb{C}}

\newcommand{\Le}{\mathsf{L}}

\newcommand{\QQ}{\mathbb{Q}}
\newcommand{\lcm}{\operatorname{lcm}}
\renewcommand{\vec}[1]{{\bf #1}}
\providecommand{\multi}[1]{\llbracket #1 \rrbracket}

\renewcommand{\pmod}[1]{\,(\operatorname{mod} #1)}
\renewcommand{\phi}{\varphi}

\newcommand{\stirling}[2]{\genfrac\{\}{0pt}{}{#1}{#2}}

\renewcommand\>{\rangle}
\newcommand\<{\langle}

\newcommand{\Mean}{\mathsf{Mean}\,}
\newcommand{\Median}{\mathsf{Median}\,}
\newcommand{\Mode}{\mathsf{Mode}\,}
\newcommand{\Var}{\mathsf{Var}\,}
\newcommand{\StDev}{\mathsf{StDev}\,}
\newcommand{\HarMean}{\mathsf{HarMean}\,}
\newcommand{\GeoMean}{\mathsf{GeoMean}\,}
\newcommand{\Skew}{\mathsf{Skew}\,}
\renewcommand{\min}{\mathsf{Min}\,}
\renewcommand{\max}{\mathsf{Max}\,}

\theoremstyle{plain}
\newtheorem{Theorem}[equation]{Theorem}
\newtheorem{Lemma}[equation]{Lemma}

\newtheorem{Corollary}[equation]{Corollary}
\theoremstyle{definition}

\newtheorem{Example}[equation]{Example}

\allowdisplaybreaks
\begin{document}
\title[Factorization length distribution for affine semigroups II]{Factorization length distribution for affine semigroups~II:\ asymptotic behavior for numerical semigroups with arbitrarily many generators}

\author[S.~Garcia]{Stephan Ramon Garcia}
\address{Department of Mathematics, Pomona College, 610 N. College Ave., Claremont, CA 91711} 
\email{stephan.garcia@pomona.edu}
\urladdr{\url{http://pages.pomona.edu/~sg064747}}

\author[M.~Omar]{Mohamed Omar}
\address{Department of Mathematics, Harvey Mudd College, 301 Platt Blvd., Claremont, CA 91711}
\email{omar@g.hmc.edu}
\urladdr{www.math.hmc.edu/~omar}
\author[C.~O'Neill]{Christopher O'Neill}
\address{Mathematics Department, San Diego State University, San Diego, CA 92182}
\email{cdoneill@sdsu.edu}
\urladdr{https://cdoneill.sdsu.edu/}

\author[S.~Yih]{Samuel Yih}
\address{UCLA Mathematics Department, Box 951555, Los Angeles, CA 90095}
\email{samyih@math.ucla.edu}
\urladdr{https://www.math.ucla.edu/people/grad/samyih}

\thanks{First author was partially supported by NSF grant DMS-1800123.}

\begin{abstract}
For numerical semigroups with a specified list of (not necessarily minimal) generators, we obtain explicit asymptotic expressions, and in some cases quasipolynomial/quasirational representations, for all major factorization length statistics.  This involves a variety of tools that are not standard in the subject, such as algebraic combinatorics (Schur polynomials), probability theory (weak convergence of measures, characteristic functions), and harmonic analysis (Fourier transforms of distributions).  We provide instructive examples which demonstrate the power and generality of our techniques.  We also highlight unexpected consequences in the theory of homogeneous symmetric functions.
\end{abstract}


\maketitle

%%%%%%%%%%%%%%%%%%%%%%%%%%%%%%%%%%%%%%%%%%%%%%%%%%%%%%%%%%%%%%%%%%%%%%%%%
\section{Introduction}%%%%%%%%%%%%%%%%%%%%%%%%%%%%%%%%%%%%%%%%%%%%%%%%%%%
\label{sec:intro}%%%%%%%%%%%%%%%%%%%%%%%%%%%%%%%%%%%%%%%%%%%%%%%%%%%%%%

In what follows, $\NN = \{0,1,2,\ldots\}$ denotes the set of nonnegative integers.  
A~\emph{numerical semigroup} $S \subset \NN$ is an additive subsemigroup containing $0$.  We write 
\begin{equation*}
S = \<n_1, n_2, \ldots, n_k\> = \{a_1n_1 + a_2 n_2 + \cdots + a_kn_k \,:\, a_1,a_2, \ldots, a_k \in \NN\}
\end{equation*}
for the numerical semigroup generated by distinct positive
$n_1 < \cdots < n_k$
in $\NN$.  
Each numerical semigroup $S$ admits a finite generating set.  Moreover,
there is a unique generating set that is minimal with respect to containment~\cite{NSBook}. 
We always assume $S$ has finite complement in $\NN$ or, equivalently, 
$\gcd(n_1,n_2, \ldots, n_k) = 1$, and that the
generators $n_1,n_2,\ldots,n_k$ are listed in increasing order.
We do not assume that $n_1,n_2, \ldots, n_k$ minimally generate $S$.  

A \emph{factorization} of $n \in S$ is an expression
\begin{equation*}
n = a_1n_1 + a_2 n_2 + \cdots + a_kn_k
\end{equation*}
of $n$ as a sum of generators of $S$, which we represent here using the $k$-tuple 
$\vec{a} = (a_1, a_2, \ldots, a_k) \in \NN^k$.  
The \emph{length} of the factorization $\vec{a}$ is 
\begin{equation*}
\|\vec{a}\| = a_1 + a_2 + \cdots + a_k.
\end{equation*}
The \emph{length multiset of $n$}, denoted $\Le\multi{n}$, is the multiset 
with a copy of $\|\vec{a}\|$ for each factorization $\vec{a}$ of $n$.  Recall that a \emph{multiset} 
is a set in which repetition is taken into account; that is, its elements can occur multiple times.  In particular, the cardinality 
$|\Le \multi{n}|$ of $\Le\multi{n}$ equals the number of factorizations of $n$.

It is well known that all sufficiently large $n \in \NN$ belong to $S$ when the generators are relatively prime.  The largest integer that does not belong to $S$, called its \emph{Frobenius number}, has been studied extensively in the literature~\cite{diophantine}.  As an extension of this, the $s$-Frobenius numbers (i.e., the largest integer with at most $s$ factorizations) has also been studied~\cite{fukshansky2011bounds}, as has an analogous question for rings of integers~\cite{gao2011quantitative}.  More recently, J.~Bourgain and Ya.~G.~Sinai~\cite{bourgainsinai}, among others~\cite{alievhenkaicke,RNS}, investigated the asymptotic behavior of the Frobenius number, as did V.I.~Arnold~\cite{arnold} in the context of estimating the number of factorizations of elements of $S$.  

Factorizations and their lengths have been studied extensively under the broad umbrella of factorization theory~\cite{schmid2009characterization,geroldinger2006non,chapman2000half} (see \cite{nonuniq} for a thorough introduction).  Investigations usually concern sets of lengths (i.e., without repetition), including asymptotic structure theorems \cite{GH92,structurethm,narkiewiczconjecture,gao2000systems} as well as specialized results spanning numerous families of rings and semigroups from number theory~\cite{integervaluedpolys,factoralgebraicintegers,acmfirst}, algebra~\cite{modulessurvey,noncommutativefactor} and elsewhere (see the survey~\cite{setsoflengthmonthly} and the references therein).  
Several combinatorially-flavored invariants have also been studied (e.g., elasticity~\cite{elasticitysurvey,geroldinger2018longelasticity}, the delta set~\cite{krulldeltaset,chapman2014delta}, and the catenary degree~\cite{geroldinger1997chains,geroldinger2011catenary}) to obtain more refined comparisons of length sets across different settings~\cite{chapman2017krull}.  Numerical semigroups have received particular attention~\cite{numericalrealization,numericalfactorsurvey,blanco2011semigroup}, in part due to their suitability for computation~\cite{numericalsgpsgap,computationoverview} and the availability of machinery from combinatorial commutative algebra~\cite{semigroupalgebra,factorhilbert} (see~\cite{CCA} for background on the latter).  Additionally, factorizations of numerical semigroup elements arise naturally in discrete optimization as solutions to knapsack problems~\cite{pisinger1998knapsack,de2013algebraic} as well as in algebraic geometry and commutative algebra~\cite{abhyankar1967local,barucci1997maximality}.  

One of the crowning achievements in factorization theory is the \emph{structure theorem for sets of length}, which in this setting states that for any numerical semigroup $S$, there exist constants $d, M > 0$ such that for all sufficiently large elements $n \in S$, the length set $L(n)$ is an arithmetic sequence from which some subset of the first and last $M$ elements are removed~\cite{nonuniq}.  As a consequence, most invariants derived from factorization length focus on extremal lengths.

% \begin{Theorem}[\!\!{\cite[Theorem~4.3.6]{nonuniq}}]\label{t:lstructure}
% Let $S = \<n_1, n_2, \ldots, n_k\>$ be a numerical semigroup.  There is an integer $M > 0$ such that for all $n \in S$, the length set $\mathsf L_S(n)$ equals an arithmetic sequence from which some subset of the first and last $M$ elements are removed.  
% \end{Theorem}

We consider here asymptotic questions surrounding length multisets of numerical semigroups.  This question was initially studied in~\cite{lengthdistribution1} for three-generated numerical semigroups, where a closed form for the limiting distribution was obtained via careful combinatorial arguments for bounding factorization-length multiplicities.  This approach proved difficult, if not impossible, when four or more generators are allowed and \cite{lengthdistribution1} ended with many questions unanswered.

Theorem~\ref{Theorem:Main} below, our main result, answers almost all questions about the asymptotic properties of important statistical quantities associated to factorization lengths
in numerical semigroups.  It relates asymptotic questions about factorization lengths to properties of an explicit
probability distribution, which permits us to obtain numerous asymptotic predictions in closed form.
Our theorem recovers the key results from \cite{lengthdistribution1} on three-generated semigroups, and generalizes them
to semigroups with an arbitrary number of generators. 

For what follows, we require some algebraic terminology.
The \emph{complete homogeneous symmetric polynomial} of degree $p$ in the
$k$ variables $x_1, x_2, \ldots, x_k$ is
\begin{equation*}
h_p(x_1,x_2,\ldots,x_k) \quad=\!\! 
\sum_{1 \leq \alpha_1 \leq \cdots \leq \alpha_{p} \leq k} x_{\alpha_1} x_{\alpha_2}\cdots x_{\alpha_p},
\end{equation*}
the sum of all degree $p$ monomials in $x_1,x_2,\ldots,x_k$.  A \emph{quasipolynomial} of degree $d$ is a function $f:\ZZ\to\CC$ of the form
\begin{equation*}
f(n) = c_d(n) n^d + c_{d-1}(n) n^{d-1} + \cdots + c_1(n) n + c_0(n),
\end{equation*}
in which the coefficients $c_1(n), c_2(n),\ldots, c_d(n)$
are periodic functions of $n$ \cite{continuousdiscretely}.  A~\emph{quasirational function} is a quotient of two quasipolynomials.
The cardinality of a set $X$ is denoted $|X|$.

\begin{Theorem}\label{Theorem:Main}
Let $S= \<n_1,n_2,\ldots,n_k\>$, in which $k\geq 3$, $\gcd(n_1,n_2,\ldots,n_k)=1$, and $n_1 < n_2 < \cdots < n_k$.
\begin{enumerate}
\item For real $\alpha < \beta$,
\begin{equation*}
\lim_{n\to\infty} \frac{ |\{ \ell \in \Le\multi{n} : \ell \in[\alpha n, \beta n]\} |}{ | \Le\multi{n} |}
 \,\,=\,\, \int_{\alpha}^{\beta} F(t)\,dt,
\end{equation*}
where $F:\RR\to\RR$ is the probability density function 
\begin{equation*}
F(x) :=\frac{(k-1)n_1n_2\cdots n_k}{2}\sum_{r=1}^k \frac{|1-n_rx|(1-n_r x)^{k-3}}{\prod_{j\neq r}(n_j-n_r)}.
\end{equation*}
The support of $F$ is $\big[\frac{1}{n_k}, \frac{1}{n_1}\big]$.  

\item For $p \in \NN$, the $p$th moment of $F$ is
\begin{equation*}
\int_0^1 t^p F(t)\,dt \,\,=\,\, \binom{p+k-1}{p}^{-1} h_p\left( \frac{1}{n_1}, \frac{1}{n_2} ,\ldots, \frac{1}{n_k} \right).
\end{equation*}

\item For any continuous function $g:(0,1)\to\CC$,
\begin{equation*}
\lim_{n\to\infty}  \frac{1}{ |\Le\multi{n}|} \sum_{\ell \in \Le\multi{n}} g\bigg(\frac{\ell}{n}\bigg) \,\,=\,\, \int_0^1 g(t) F(t)\,dt.
\end{equation*}
\end{enumerate}
\end{Theorem}

In Theorem \ref{Theorem:Main}a, observe that $F$ is a piecewise-polynomial function of degree $k-2$ that is $(k-3)$-times continuously differentiable, but not everywhere differentiable $k-2$ times.  In particular, its smoothness increases as the number of generators increases.
This is characteristic of the Curry--Schoenberg B-spline from computer-aided design \cite{Curry}, of which the function $F$ is a special case;
this connection is discussed in much greater detail in \cite{Bottcher}.

The explicit nature and broad generality of Theorem \ref{Theorem:Main} permit strikingly accurate asymptotic predictions, often in closed form, of virtually every statistical quantity
related to factorization lengths when considered with multiplicity.  For example, Theorem \ref{Theorem:Main} immediately predicts the number of factorizations of $n$,
the moments of the factorization-length multiset $\Le\multi{n}$, its mean, standard deviation, median, mode, skewness, and so forth (see Section \ref{Section:Applications}).
The flexibility afforded by Theorem \ref{Theorem:Main}c permits us to address quantities such as the harmonic and geometric mean factorization length,
which would previously have been beyond the scope of standard semigroup-theoretic techniques.  

The proof of Theorem \ref{Theorem:Main} is contained in Section \ref{Section:ProofMain}.
It involves a variety of tools that are not standard fare in the numerical semigroup literature.  
For example, weak convergence of probability measures, Fourier transforms of distributions, and the theory of characteristic functions come into play.  
In addition, Theorem~\ref{Theorem:Main} builds upon two other results, described below, whose origins are in complex variables (Theorem~\ref{Theorem:Moment})
and algebraic combinatorics (Theorem~\ref{Theorem:Combo}).

Theorem~\ref{Theorem:Moment}, whose proof is deferred until Section~\ref{Section:ProofMoment}, 
concerns a quasipolynomial representation for the $p$th power sum of the factorization lengths of $n$ (the main ingredient for the $p$th moment of $F(x)$). 
Although this result is of independent interest to the numerical semigroup community, its true power emerges when
combined with Theorems~\ref{Theorem:Main} and~\ref{Theorem:Combo}.

\begin{Theorem}\label{Theorem:Moment}
Let $S= \<n_1,n_2,\ldots,n_k\>$, in which $k\geq 3$, $\gcd(n_1,n_2,\ldots,n_k)=1$, and $n_1 < n_2 < \cdots < n_k$.
For $p \in \NN$, 
\begin{equation*}
\sum_{\ell \in \Le\multi{n}} \ell^p = \frac{p!}{(k + p - 1)! (n_1 n_2 \cdots n_k)} h_p \bigg( \frac{1}{n_1}, \frac{1}{n_2}, \ldots, \frac{1}{n_k} \bigg) n^{k+p-1} + w_p(n), 
\end{equation*}
in which $w_p(n)$ is a quasipolynomial of degree at most $k+p-2$ whose coefficients have period dividing $\lcm(n_1,n_2,\ldots,n_k)$.  
\end{Theorem}

Our next result, whose proof is in Section~\ref{Section:Combo},
is an exponential generating function identity.  
Although its derivation involves a bit of algebraic combinatorics and 
the result itself might seem a bit of a digression, 
this identity is a crucial ingredient to the proof of Theorem~\ref{Theorem:Main}.

\begin{Theorem}\label{Theorem:Combo}
Let $x_1,x_2,\ldots,x_k \in \CC\backslash\{0\}$ be distinct.
For $z \in \CC$,
\begin{equation*}
\sum_{p=0}^{\infty} \frac{h_p(x_1,x_2,\ldots,x_k)}{(p+k-1)!} z^{p+k-1}
\,\,=\,\, \sum_{r=1}^k \frac{e^{x_r z}}{\prod_{j\neq r}(x_r-x_j)}.
\end{equation*}
\end{Theorem}

There are several unexpected consequences of our work 
to the realm of symmetric functions.  For example, Theorem \ref{Theorem:CHS} in Section \ref{Section:CHS}
provides a novel probabilistic interpretation
of the complete homogeneous symmetric polynomials.  This not only recovers a 
well-known positivity result (Corollary \ref{Corollary:Hunter}), it also provides a natural method to 
extend the definition of $h_p(x_1,x_2,\ldots,x_k)$ to nonintegral $p$.

We are optimistic that Theorem~\ref{Theorem:Main} will prove to be a standard tool in the study of numerical semigroups;
statistical results about factorization lengths that before appeared intractable are now straightforward consequences of
Theorem~\ref{Theorem:Main}.  We devote all of Section~\ref{Section:Applications} to applications and examples
of our results.  
Sections~\ref{Section:ProofMoment},~\ref{Section:Combo}, and~\ref{Section:ProofMain} contain the proofs of Theorems~\ref{Theorem:Moment},~\ref{Theorem:Combo}, and~\ref{Theorem:Main} respectively.  
We wrap up in Section~\ref{Section:End} with some closing remarks.

%%%%%%%%%%%%%%%%%%%%%%%%%%%%%%%%%%%%%%%
\section{Applications and Examples}\label{Section:Applications}

This section consists of a host of examples and applications of Theorems~\ref{Theorem:Main} and~\ref{Theorem:Moment}.  
We avoid the traditional corollary-proof format, which would soon become overbearing, in favor of a more leisurely and less staccato pace.
In particular, we demonstrate how a wide variety of factorization-length statistics, some frequently considered and others more
exotic, can be examined using our methods.  The following examples and commentary illustrate the effectiveness of our techniques as well as their implementation.

We begin in Subsection~\ref{Subsection:Stats} with a brief rundown of fundamental factorization-length statistics,
giving closed-form formulas for the asymptotic behavior when convenient.  
In Subsection~\ref{Subsection:Three},
we recover all of the key results of~\cite{lengthdistribution1} on three-generator numerical semigroups.
Subsection~\ref{Subsection:Four} contains explicit formulas, all of them novel, for asymptotic statistics in four-generated semigroups.
Numerical semigroups with more generators and related phenomena are discussed in Subsection~\ref{Subsection:More}.

%%%%%%%%%%%%%%%%%%%%
\subsection{Factorization-length statistics}\label{Subsection:Stats}
Fix $S = \<n_1, n_2, \ldots, n_k\>$, where as always we assume that $\gcd(n_1,n_2,\ldots,n_k)=1$.
The quasipolynomial or quasirational functions mentioned below all have 
$\QQ$-valued coefficients with periods dividing $\lcm(n_1,n_2,\ldots,n_k)$.  
For each key factorization-length statistic we provide an explicit, asymptotically
equivalent expression when available.  We say that $f(n) \sim g(n)$ if $\lim_{n\to\infty} f(n)/g(n) = 1$ and 
$f(n) = O(g(n))$ if there is a constant $C$ such that $|f(n)| \leq C|g(n)|$ for sufficiently large $n \in \NN$.

\begin{enumerate}[leftmargin=*]
\item \textbf{Number of Factorizations.}
Theorem \ref{Theorem:Moment} with $p=0$ implies that
the cardinality $|\Le\multi{n}|$ of the factorization length multiset $\Le\multi{n}$ is a quasipolynomial and
\begin{equation}\label{eq:Znk}
|\Le\multi{n}| \,\,=\,\,  \frac{n^{k-1}}{(k - 1)! (n_1 n_2 \cdots n_k)} + O(n^{k-2}).
\end{equation}

\item \textbf{Moments.}
Theorem \ref{Theorem:Moment} and \eqref{eq:Znk} imply that
the \emph{$p$th factorization length moment}
\begin{equation}\label{eq:mpn}
m_p(n)\,\,:=\,\, \frac{1}{|\Le\multi{n}|}\sum_{\ell \in \Le\multi{n}}\!\ell^p 
\,\,\, \sim \,\,  \binom{p+k-1}{p}^{\!-1}\!\! h_p\bigg( \frac{1}{n_1}, \frac{1}{n_2} ,\ldots, \frac{1}{n_k} \bigg) n^p
\end{equation}
is quasirational.  
 
\item \textbf{Mean.}
The preceding implies that the \emph{mean factorization length}
\begin{equation*}
m_1(n)\,\,:=\,\,\frac{1}{|\Le\multi{n}|} \sum_{\ell \in \Le\multi{n}} \ell \,\,
\sim \,\, \frac{n}{k}\left( \frac{1}{n_1} + \frac{1}{n_2} + \cdots + \frac{1}{n_k}\right) 
\end{equation*}
is quasirational.  It is asymptotically linear as $n \to \infty$ and 
its slope is the reciprocal of the harmonic mean of the generators of $S$.  

\item \textbf{Variance and standard deviation.}
The factorization length \emph{variance}, given by
$\sigma^2(n) := m_2(n) - (m_1(n))^2$, is quasirational by (b).
From~\eqref{eq:mpn}, we have
\begin{equation*}
\sigma^2(n) 
\,\,\sim\,\, \frac{n^2}{k^2(k+1)} \left( (k-1) \sum_{i=1}^k \frac{1}{n_i^2} - 2 \sum_{i < j} \frac{1}{n_in_j} \right ).
\end{equation*}
The \emph{standard deviation} is then $\sigma(n)$, the square root of the variance.

\item \textbf{Median.}
Theorem \ref{Theorem:Main}a ensures that the \emph{median} factorization length satisfies
\begin{equation*}
\Median \Le\multi{n} \,\,\sim\,\, \beta n ,
\end{equation*}
in which $\beta \in [0,1]$ is the unique positive real number so that $\int_0^{\beta} F(t)\,dt = \frac{1}{2}$.  Note that it was already demonstrated in~\cite{lengthdistribution1} that $\beta$ can be irrational, even when $k = 3$, in which case the median cannot not be quasirational in $n$.  

\item \textbf{Mode.}
Since the function $F$ is known to be unimodal (see \cite[Thm.~1]{Curry} and \cite{Bottcher}), the \emph{mode} factorization length satisfies
\begin{equation*}
\Mode \Le\multi{n} \,\,\sim\,\, n \operatorname{arg max} F(x),
\end{equation*}
in which $\operatorname{arg max} F(x)$ is the unique value in $[0,1]$ at which $F$
assumes its absolute maximum. 

\item \textbf{Skewness.}
The factorization length \emph{skewness} is
\begin{equation*}
\Skew \Le\multi{n} 
:= \frac{1}{| \Le\multi{n}|} \sum_{\ell \in \Le\multi{n}} \left( \frac{\ell - m_1(n)}{\sigma(n)} \right)^3
= \frac{m_3(n) - 3m_1(n) \sigma^2(n) - m_1(n)^3}{\sigma^3(n)} ,
\end{equation*}
the third centered moment.  In light of (b), (c), and (d), an explicit asymptotic formula for $\Skew \Le\multi{n}$
can be given, although we refrain from doing so.  

\item \textbf{Min / Max.}
The minimum and maximum factorization lengths satisfy
\begin{equation*}
\max \Le\multi{n}  \,\,\sim\,\, \frac{n}{n_1}
\qquad \text{and} \qquad
\min \Le\multi{n} \,\,\sim\,\, \frac{n}{n_k}.
\end{equation*}
This follows from Theorem \ref{Theorem:Main}a since the distribution $F(t)$ is supported on $[1/n_k,1/n_1]$ and places
mass on any open neighborhood of its endpoints (it is known that $\max \Le\multi{n}$ and $\min \Le\multi{n}$ are linear quasipolynomials with leading coefficients $1/n_1$ and $1/n_k$,
respectively \cite[Theorems~4.2 and~4.3]{minmaxquasi}).

\item \textbf{Harmonic mean.} 
The \emph{harmonic mean} factorization length satisfies
\begin{equation*}
H(n)\,\,:=\,\, \frac{| \Le\multi{n} |}{\sum_{\ell \in \Le\multi{n}} \ell^{-1}}
\,\, \sim \,\,  \frac{n}{ \int_0^1 t^{-1} F(t)\, dt}.
\end{equation*}
The integral is taken over $[0,1]$ for convenience; 
since $F$ is supported on $[1/n_k,1/n_1]$, the integrand vanishes at $t=0$.

\item \textbf{Geometric mean.}
The \emph{geometric mean} factorization length satisfies
\begin{equation*}
G(n) \,\,:=\,\, \Big( \prod_{\ell \in \Le\multi{n}} \ell \Big)^{\frac{1}{|\Le\multi{n}|}} \,\,\sim\,\, n e^{\int_0^1 (\log t) F(t)\,dt}
\end{equation*}
since
\begin{align*}
\log G(n)
&= \frac{1}{| \Le \multi{n} | } \sum_{\ell \in \Le\multi{n}} \log \ell
= \frac{1}{| \Le \multi{n} | } \sum_{\ell \in \Le\multi{n}} \bigg( \log \frac{\ell}{n} + \log n\bigg) \\
&= \log n +  \frac{1}{| \Le \multi{n} | } \sum_{\ell \in \Le\multi{n}}  \log \frac{\ell}{n} 
\,\,\sim\,\, \log n + \int_0^1 (\log t) F(t)\,dt.
\end{align*}
\end{enumerate}

For the sake of uniformity, we often prefer to use the more explicit notation
$\Mean \Le\multi{n}$,
$\Median \Le\multi{n}$,
$\Mode \Le\multi{n}$,
$\Var \Le\multi{n}$,
$\StDev \Le\multi{n}$,
$\HarMean \Le\multi{n}$, 
$\Skew \Le\multi{n}$, and
$\GeoMean \Le\multi{n}$,
instead of distinctive symbols, such as $\mu(n)$ or $\sigma(n)$.

%%%%%%%%%%%%%%%%%%%%%%%%%%%%%%%%%%%%%%%%%%%%%
\subsection{Three generators: triangular distribution}\label{Subsection:Three}
\begin{figure}
    \centering
    \begin{tikzpicture}[thick,scale=0.6, every node/.style={scale=0.9}]
        \draw[->,thin](0,0)--(17,0);
        \draw[->,thin](0,0)--(0,5.5);
      
        \draw[thick,black](1,0)to(6,5);
        \draw[thick,black](6,5)to(15,0);
        \draw[fill=black](1,0) circle (.05cm) node[below,yshift=-2pt]{$\frac{1}{n_3}$};
        \draw[fill=black](15,0) circle (.05cm) node[below,yshift=-2pt]{$\frac{1}{n_1}$};
        \draw[fill=black](6,5) circle (.05cm);

        \draw[thin] (-0.15,5)node[left]{$\frac{2n_1n_3}{n_3-n_1}$}--(0.15,5);

        \draw[thin] (6,-0.15)node[below]{$\frac{1}{n_2}$}--(6,0.15);
        \draw[thin,dashed] (6,0)--(6,5);
        
    \end{tikzpicture}
    \caption{The asymptotic length distribution function $F(x)$ for a three-generated semigroup $S = \<n_1,n_2,n_3\>$ is a
    triangular distribution on $[1/n_3, 1/n_1]$ with peak of height $2n_1n_3 / (n_3-n_1)$ at $1/n_2$.}
    \label{Figure:Triangle}
\end{figure}
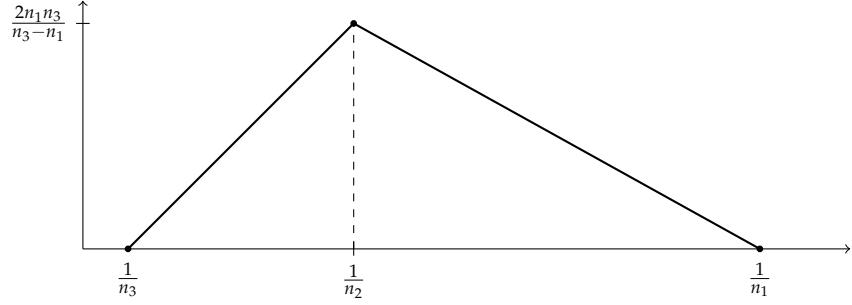	

The asymptotic behavior of factorization lengths in three-generator semigroups
was studied in \cite{lengthdistribution1} with other methods.
Theorem \ref{Theorem:Main} recovers all of the main results from that paper.

For $S = \<n_1,n_2,n_3\>$, the function $F(x)$ of Theorem \ref{Theorem:Main} is a triangular distribution; see Figure \ref{Figure:Triangle}.  Indeed,
letting $k=3$ and $(a,b,c) = (\frac{1}{n_3}, \frac{1}{n_1}, \frac{1}{n_2})$ in Theorem \ref{Theorem:Main} we obtain
\begin{equation}\label{eq:TriangleF}
F(x) 
=
\begin{cases}
0 & \text{if $x \leq a$},\\
\dfrac{2(x-a)}{(b-a)(c-a)} & \text{for $a \leq x \leq c$}, \\[10pt]    
\dfrac{2(b-x)}{(b-a)(b-c)} & \text{for $c < x \leq b$},\\
0 & \text{if $x \geq b$}.
\end{cases}
\end{equation}
This is the familiar triangular distribution on $[a,b]$ with peak of height $2/(b-a)$
at $c \in (a,b)$ \cite[Ch.~40]{Evans}, \cite[Ch.~1]{Kotz}.
As predicted in the comments after Theorem~\ref{Theorem:Main}, the distribution function
is continuous but not everywhere differentiable. 
The standard properties of the triangular distribution provide us with the asymptotic behavior
of lengths in three-generated semigroups:
\begin{align*}
\Mean \Le\multi{n}&\,\,\sim\,\, \frac{n}{3}\left( \frac{1}{n_1} + \frac{1}{n_2} + \frac{1}{n_3} \right)  ,\\[5pt]
\Median \Le\multi{n} &\,\,\sim\,\, 
n \cdot \footnotesize
\begin{cases}\displaystyle
\frac{1}{n_3} + \sqrt{ \frac{1}{2} \left( \frac{1}{n_1} - \frac{1}{n_3}\right)\left( \frac{1}{n_2} - \frac{1}{n_3} \right)}
&\displaystyle \text{if $\frac{1}{n_2} \geq \frac{1}{2} \left( \frac{1}{n_1} + \frac{1}{n_3} \right)$},\\[10pt]
\displaystyle\frac{1}{n_1} - \sqrt{ \frac{1}{2} \left( \frac{1}{n_1} - \frac{1}{n_3}\right)\left( \frac{1}{n_1} - \frac{1}{n_2} \right)}
&\displaystyle \text{if $\frac{1}{n_2} < \frac{1}{2} \left( \frac{1}{n_1} + \frac{1}{n_3} \right)$},\\
\end{cases}
\\[5pt]
\Mode \Le \multi{n} &\,\,\sim\,\, \frac{n}{n_2},  \\[5pt]
\Var \Le \multi{n} &\,\,\sim\,\, \frac{n^2}{18}\left( \frac{1}{n_1^2} + \frac{1}{n_2^2} + \frac{1}{n_3^2} - \frac{1}{n_1 n_2} - \frac{1}{n_2 n_3}- \frac{1}{n_3 n_1} \right), \quad \text{and} \\[5pt]
\Skew \Le \multi{n}
&\,\,\sim\,\,
\frac{\sqrt{2}\Big(\frac{1}{n_1} + \frac{1}{n_3} - \frac{2}{n_2}\Big)
\Big(\frac{2}{n_1} - \frac{1}{n_3} - \frac{1}{n_2}\Big)\Big(\frac{1}{n_1} - \frac{2}{n_3} + \frac{1}{n_2}\Big)}
{5\Big( \frac{1}{n_1^2} +  \frac{1}{n_2^2} +  \frac{1}{n_3^2} - \frac{1}{n_1n_2} - \frac{1}{n_1 n_3} - \frac{1}{n_2 n_3} \Big)^{3/2} }.
\end{align*}
The harmonic and geometric means can also be worked out in closed form;
the interested reader may wish to pursue the matter further.

\begin{Example}
Consider the McNugget semigroup $S = \<6,9,20\>$.  The normalized histogram of the length multiset $\Le \multi{n}$
rapidly approaches the corresponding triangular distribution with parameters $(a,b,c) = (\frac{1}{20}, \frac{1}{9}, \frac{1}{6})$;
see Figure \ref{Figure:McNugget}.
The asymptotic formulae furnished by our results perform admirably in estimating key factorization-length statistics; see
Table \ref{Table:McNugget}.
\end{Example}

\begin{figure}
\centering
		\begin{subfigure}[b]{0.475\textwidth}
	                \centering
	                \includegraphics[width=\textwidth]{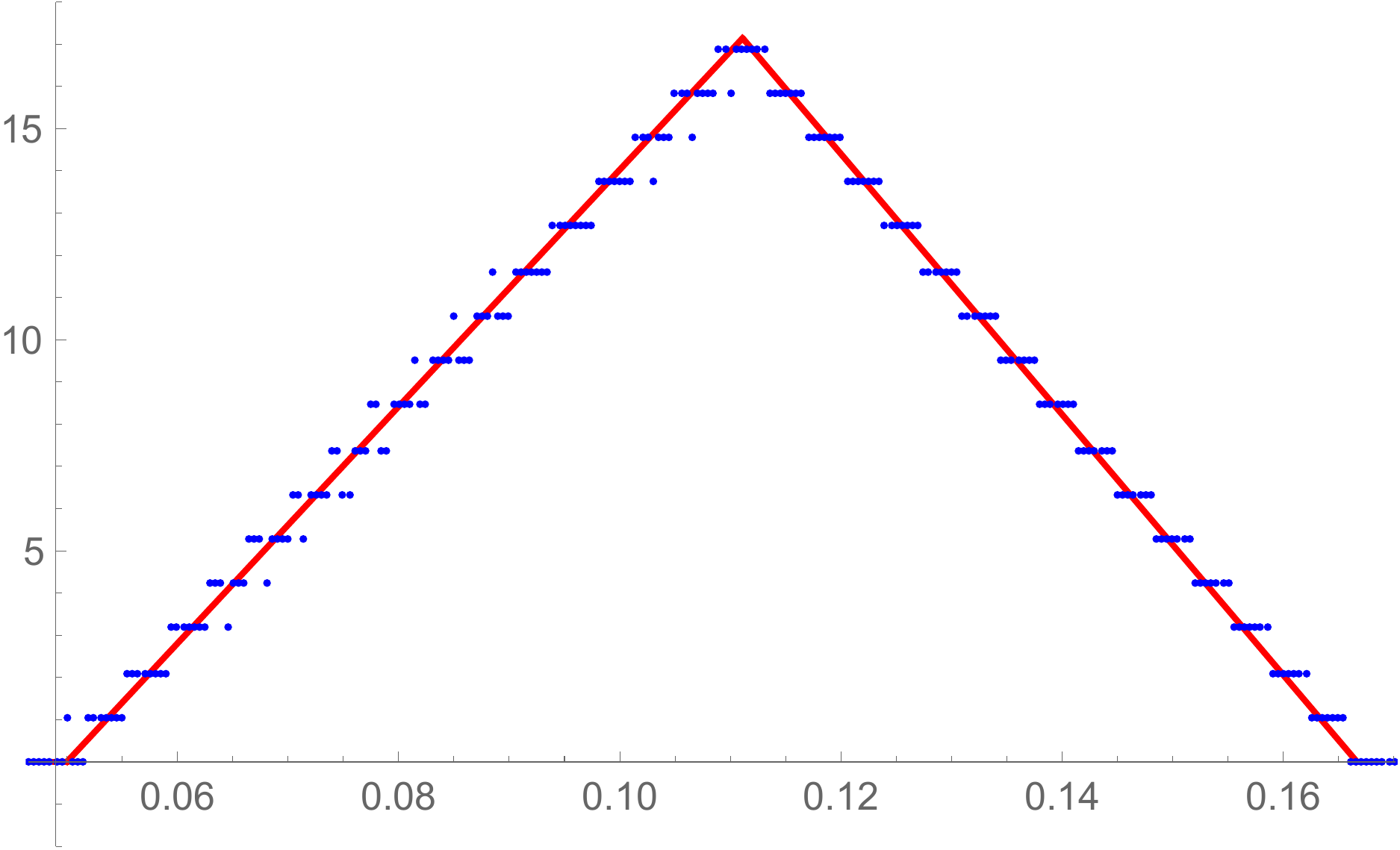}
	                \caption{$n=2{,}000$}
	        \end{subfigure}
		\begin{subfigure}[b]{0.475\textwidth}
	                \centering
	                \includegraphics[width=\textwidth]{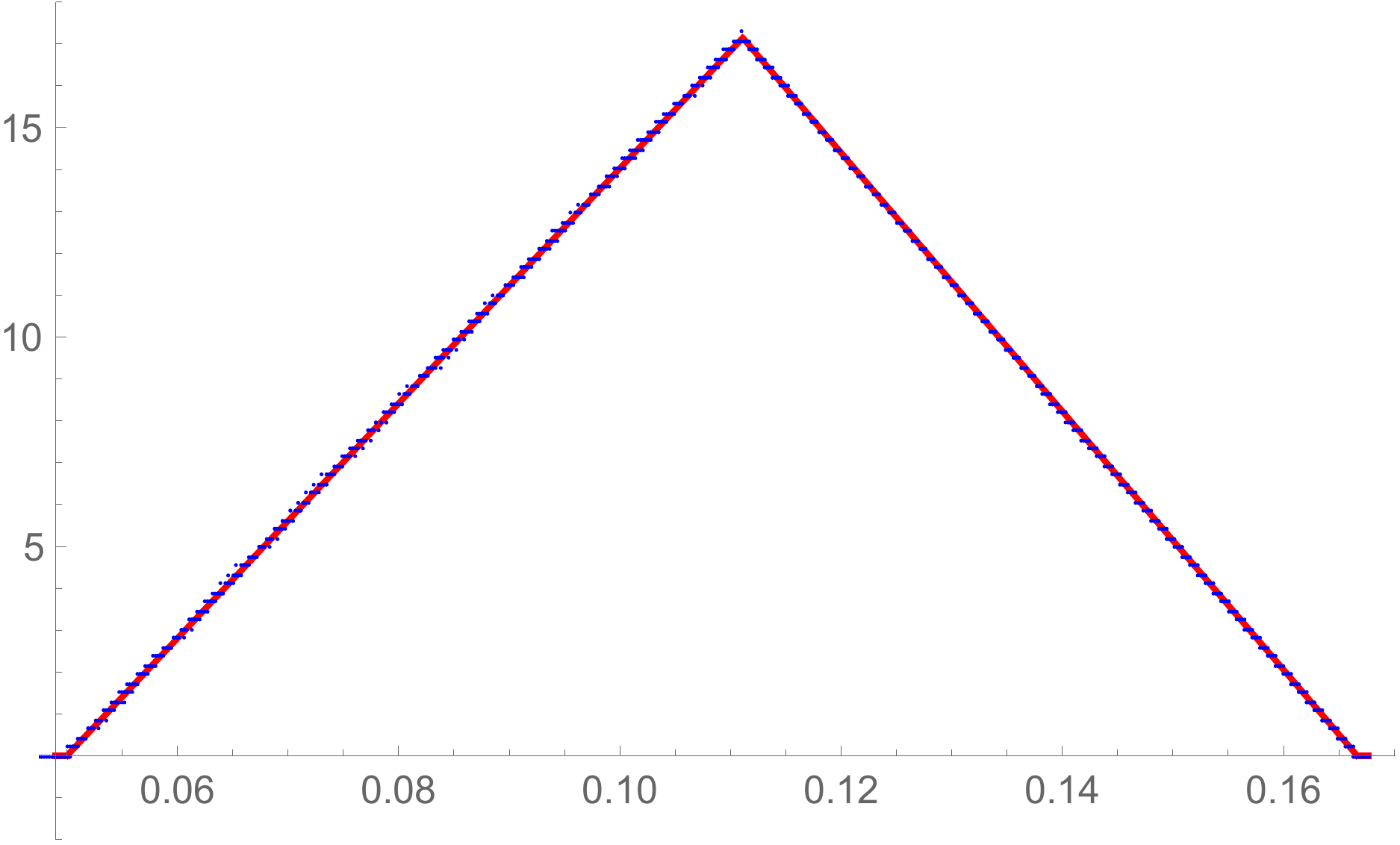}
	                \caption{$n=10{,}000$}
	        \end{subfigure}

\caption{Normalized histogram of the length multiset  $\Le\multi{n}$ (blue) and graph of the length distribution function $F(x)$ (red) for $S =  \<6,9,20\>$.  For $i \in \NN$, a blue dot occurs above $i/n$ at height equal to the multiplicity of $i$ in $\Le\multi{n}$.}
\label{Figure:McNugget}
\end{figure}

\begin{table}\footnotesize
\begin{equation*}
\begin{array}{c|cc||c|cc}
\text{Statistic} & \text{Actual} & \text{Predicted} & \text{Statistic} & \text{Actual} & \text{Predicted} \\
\hline
\Mean \Le\multi{10^5} & 10925.14 &10925.93 & \HarMean \Le\multi{10^5} & 10359.00 & 10359.86\\[3pt]
\Median \Le\multi{10^5} & 10970 & 10970.61 & \GeoMean \Le\multi{10^5} & 10650.22 & 10651.03\\[3pt]
\Mode\Le\multi{10^5} &\scalebox{0.85}{\{11109,11110,11111\} }& 11111.11 & \Skew \Le\multi{10^5} & -0.046593 & -0.046592 \\[3pt]
\StDev\Le\multi{10^5} & 2382.40 & 2382.35 & \min/\max \Le\multi{10^5} &\scalebox{0.85}{5000/16662} & \scalebox{0.85}{5000.00/16666.67}
\end{array}
\end{equation*}
\caption{Actual versus predicted statistics (rounded to two decimal places) 
for $\Le\multi{10^5}$, the multiset of factorization lengths of $100{,}000$, in 
$S = \<6,9,20\>$.}
\label{Table:McNugget}
\end{table}

%%%%%%%%%%%%%%%%%%%%%%%%%%%%%
\subsection{Four generators: piecewise quadratic}\label{Subsection:Four}
For $S = \<n_1,n_2,n_3,n_4\>$, the function $F(x)$ of Theorem \ref{Theorem:Main} is piecewise quadratic
and can be worked out in closed form:
\begin{align*}
F(x)
&=3 n_1 n_2 n_3 n_4\times 
\begin{cases}
0 & \text{if $x < \frac{1}{n_4}$},\\[5pt]
\dfrac{(1-n_4 x)^2}{(n_4-n_1) (n_4-n_2) (n_4-n_3)} & \text{if $\frac{1}{n_4} \leq x \leq \frac{1}{n_3}$},\\[12pt]
 f(x)
 & \text{if $\frac{1}{n_3} \leq x \leq \frac{1}{n_2}$},\\[4pt]
\dfrac{(1-n_1 x)^2}{(n_2-n_1) (n_3-n_1) (n_4 -n_1)} & \text{if $\frac{1}{n_2} \leq x \leq \frac{1}{n_1}$},\\[8pt]
0 & \text{if $x > \frac{1}{n_1}$},
\end{cases}
\end{align*}
in which
{\small
\begin{equation*}
f(x)=\frac{(n_1 n_2 n_3 +n_1n_2n_4-n_1 n_3 n_4-n_2 n_3 n_4) x^2-2( n_1 n_2- n_3 n_4) x
+(n_1+n_2-n_3-n_4)}{(n_3-n_1) (n_3-n_2) (n_4-n_1) (n_4-n_2)}.
\end{equation*}
}%
We remark that this explicit formula for the length distribution function completely answers the open problem suggested at the end of \cite{lengthdistribution1}.
As predicted by the comments after Theorem \ref{Theorem:Main}, $F$ is continuously differentiable but not twice differentiable. 
Moreover, one can see that $F$ is unimodal and that its absolute maximum is attained in $( \frac{1}{n_3}, \frac{1}{n_2})$.
A few computations reveal that
\begin{align*}
\Mean\Le\multi{n} 
&\,\, \sim \,\, \frac{n}{4}\left( \frac{1}{n_1} + \frac{1}{n_2} + \frac{1}{n_3} + \frac{1}{n_4}\right) ,\\[3pt]
\Mode \Le\multi{n} &\,\, \sim\,\,
\left(\frac{n_1 n_2-n_3 n_4}{n_1 n_2 n_3+n_1 n_2 n_4 - n_1 n_3 n_4-n_2 n_3 n_4}\right)n, \quad \text{and}\\[3pt]
\Var \Le \multi{n} &\,\,\sim\,\, \frac{n^2}{80}\left( 4\sum_{i-1}^4 \frac{1}{n_i^2} - 2 \sum_{i<j} \frac{1}{n_i n_j}\right).
\end{align*}
The asymptotic median factorization length is not so amenable to closed-form expression, although it is easily 
computed for specific semigroups as we see below.  

\begin{Example}
For $S = \<11,34,35,36\>$, we have 
\begin{equation*}
F(x) = 1413720
\begin{cases}
0 & \text{if $x \leq \frac{1}{36}$},\\
\frac{1}{50} (36 x-1)^2 & \text{if $\frac{1}{36} \leq x \leq \frac{1}{35}$},\\[5pt]
\frac{1}{600} \left(-15073 x^2+886 x-13\right) & \text{if $\frac{1}{35} \leq x \leq \frac{1}{34}$},\\[5pt]
\frac{1}{13800}(11 x-1)^2 & \text{if $\frac{1}{34} \leq x \leq \frac{1}{11}$},\\[5pt]
0 & \text{if $x > \frac{1}{11}$};
\end{cases}
\end{equation*}
see Figure \ref{Figure:11343536}. 
Elementary computation confirms that the median of the distribution function $F(x)$ occurs in
$[\frac{1}{34}, \frac{1}{11}]$.  For $x \in [\frac{1}{34}, \frac{1}{11}]$, we find that
\begin{equation*}
\int_0^x F(t)\,dt = 
\frac{11}{115} (43197 x^3-11781 x^2+1071 x-22)
\end{equation*}
attains the value $\frac{1}{2}$ at precisely one point, namely
$\frac{1}{11}(1-\sqrt[3]{\frac{115}{714}}) \approx 0.041$.
Thus,
\begin{equation*}
\Median  \Le\multi{n} \,\, \sim \,\, \frac{1}{11}\left(1-\sqrt[3]{\frac{115}{714}}\right) n.
\end{equation*}
Table \ref{Table:11343536} provides factorization-length statistics for $\Le\multi{10^5}$
and the strikingly accurate approximations furnished by our results.  
\end{Example}

\begin{figure}
\centering
		\begin{subfigure}[b]{0.475\textwidth}
	                \centering
	                \includegraphics[width=\textwidth]{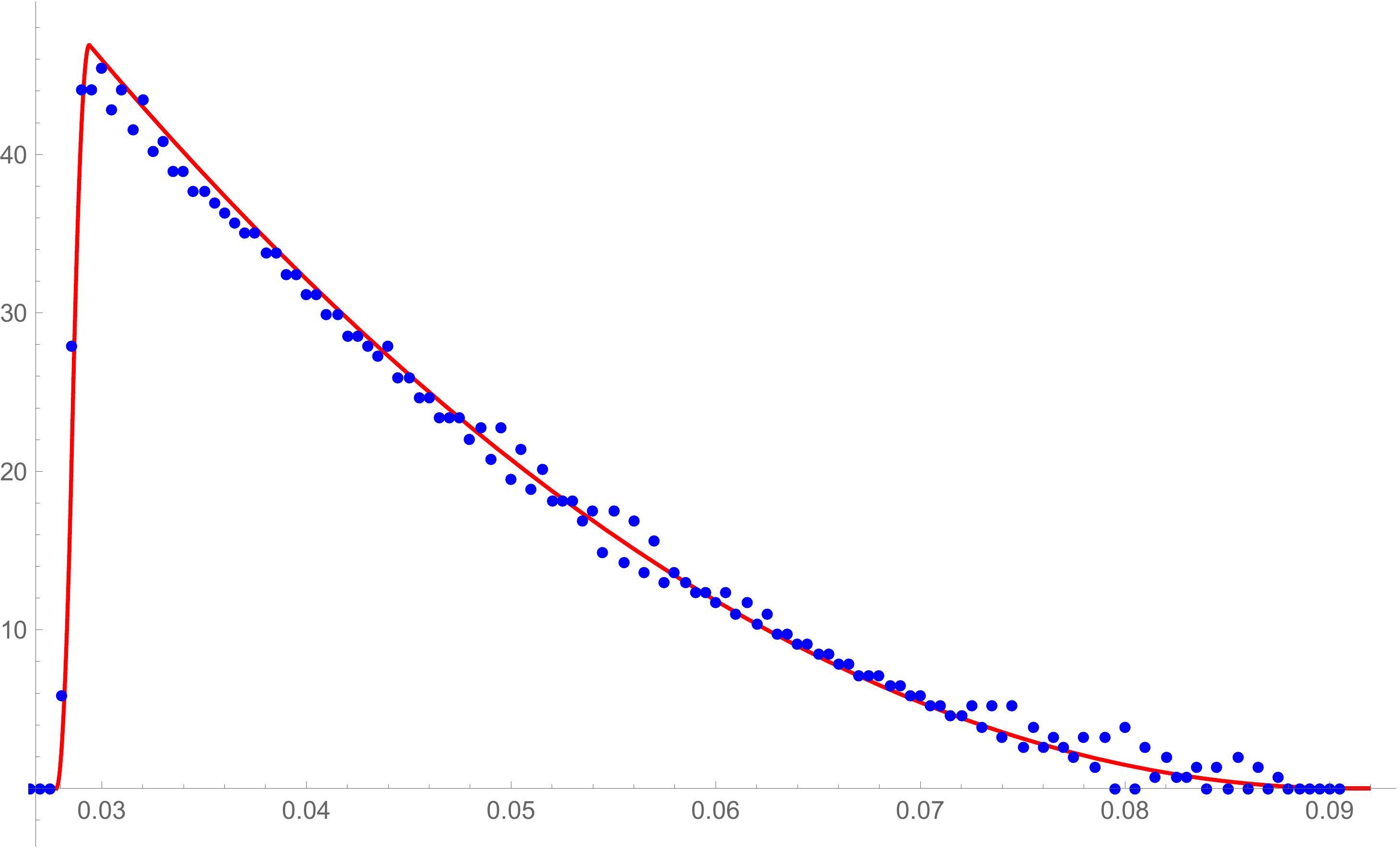}
	                \caption{$n=2{,}000$}
	        \end{subfigure}
		\begin{subfigure}[b]{0.475\textwidth}
	                \centering
	                \includegraphics[width=\textwidth]{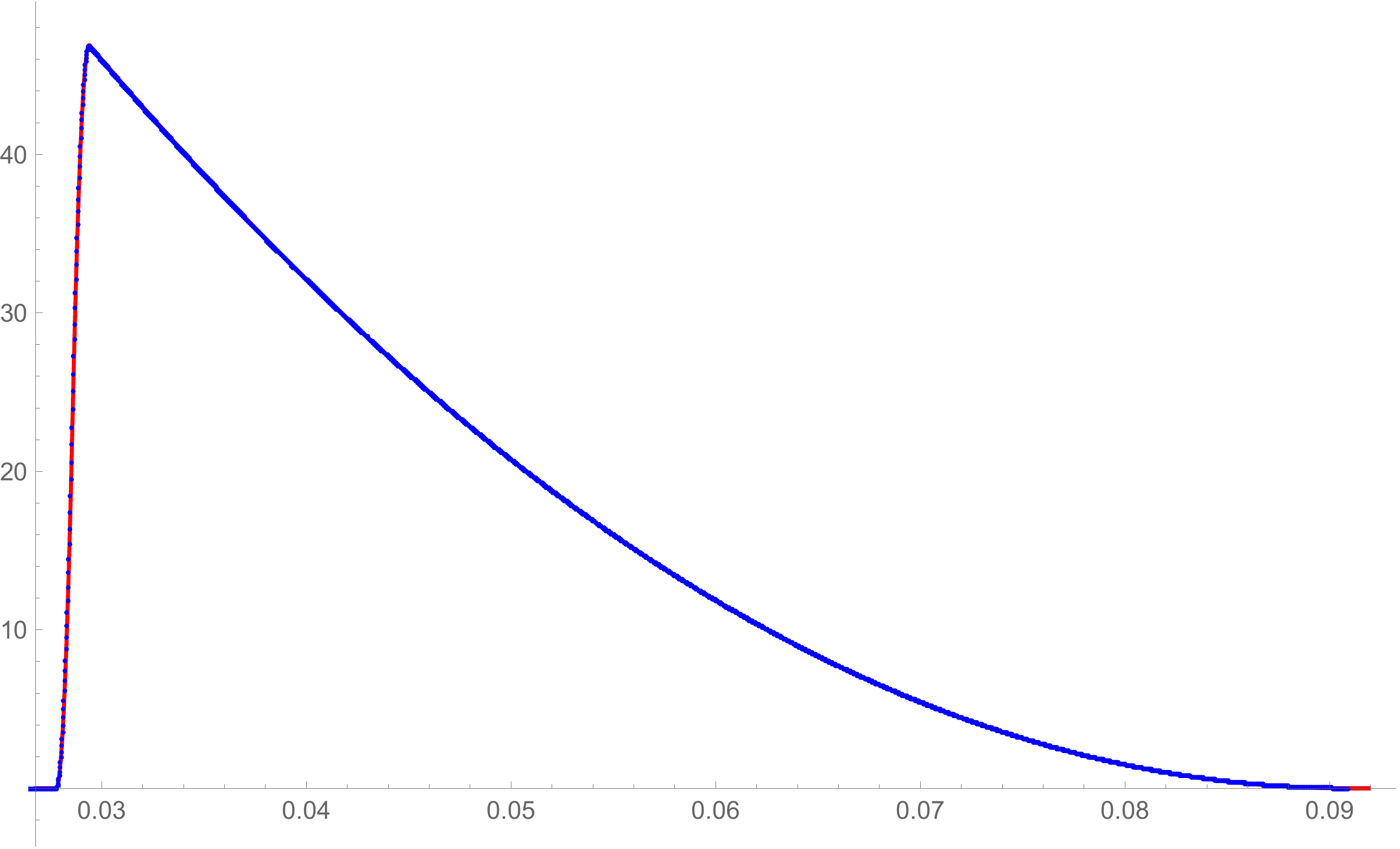}
	                \caption{$n=50{,}000$}
	        \end{subfigure}

\caption{Normalized histogram of the length multiset  $\Le\multi{n}$ (blue) and graph of the length distribution function $F(x)$ (red) for $S =  \<11,34,35,36\>$.  For each $i \in \NN$ a blue dot occurs above $i/n$ at height equal to the multiplicity of $i$ in $\Le\multi{n}$.}
\label{Figure:11343536}
\end{figure}

\begin{table}\footnotesize
\begin{equation*}
\begin{array}{c|cc||c|cc}
\text{Statistic} & \text{Actual} & \text{Predicted} &\text{Statistic} & \text{Actual} & \text{Predicted} \\
\hline
\Mean \Le\multi{10^5} & 4417.31 & 4416.76 &\HarMean \Le\multi{10^5} & 4130.30  & 4130.03\\[3pt]
\Median \Le\multi{10^5} & 4145 & 4144.69& \GeoMean \Le\multi{10^5} &  4266.46& 4266.06\\[3pt]
\Mode\Le\multi{10^5} &2939& 2939.03 & \Skew \Le\multi{10^5} & 0.8594802 & 0.8594804 \\[3pt]
\StDev\Le\multi{10^5} &  1207.84&1207.14 & \min / \max \Le \multi{10^5} &\scalebox{0.9}{2778/9082} & \scalebox{0.9}{2777.78/9090.91}\\[3pt]
\end{array}
\end{equation*}
\caption{Actual versus predicted statistics (rounded to two decimal places) 
for $\Le\multi{10^5}$, the multiset of factorization lengths of $100{,}000$, in 
$S = \<11,34,35,36\>$.}
\label{Table:11343536}
\end{table}

\begin{figure}
		\begin{subfigure}[b]{0.475\textwidth}
	                \centering
	                \includegraphics[width=\textwidth]{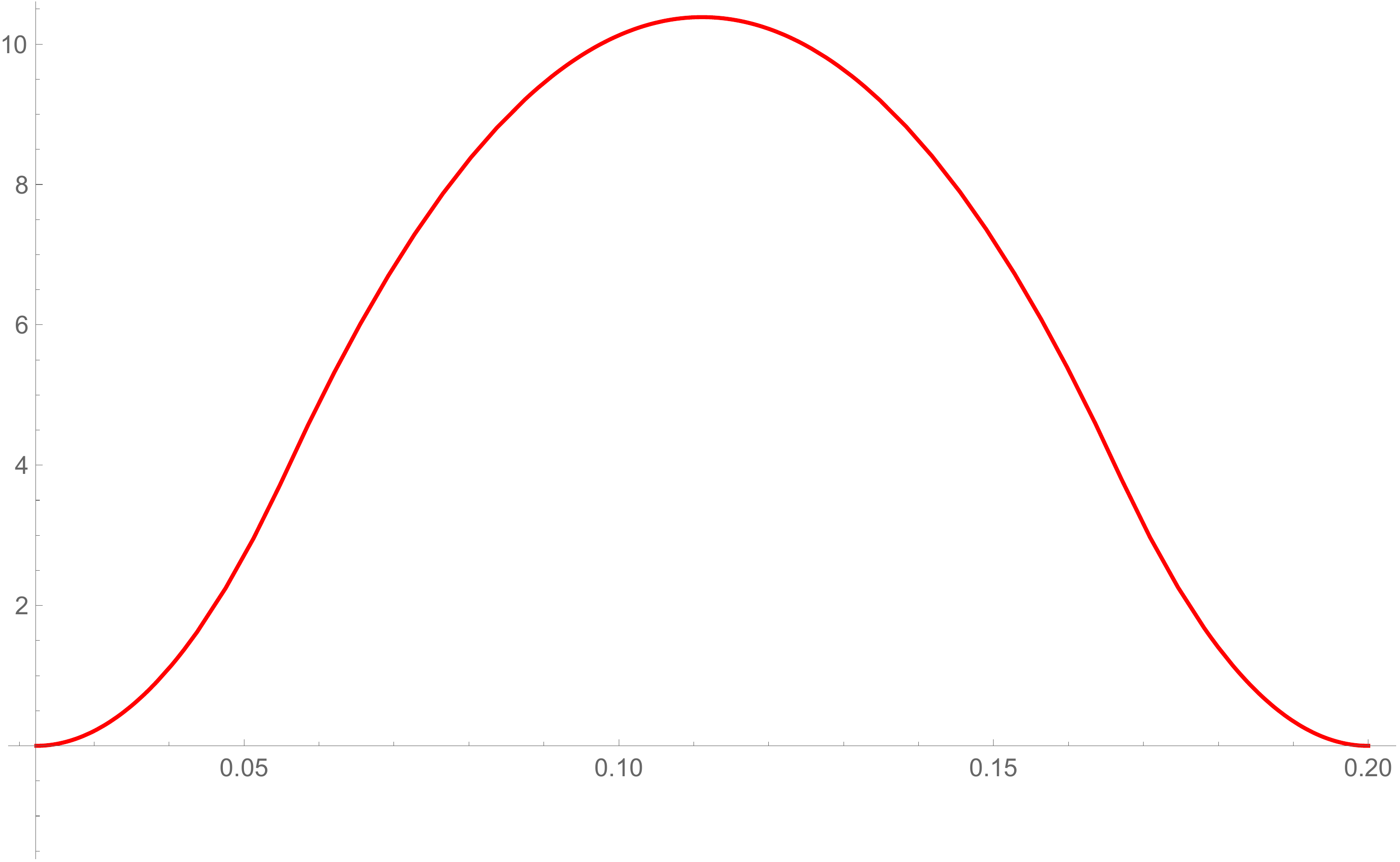}
	                \caption{$S=\<5,6,18,45\>$}
	        \end{subfigure}
		\begin{subfigure}[b]{0.475\textwidth}
	                \centering
	                \includegraphics[width=\textwidth]{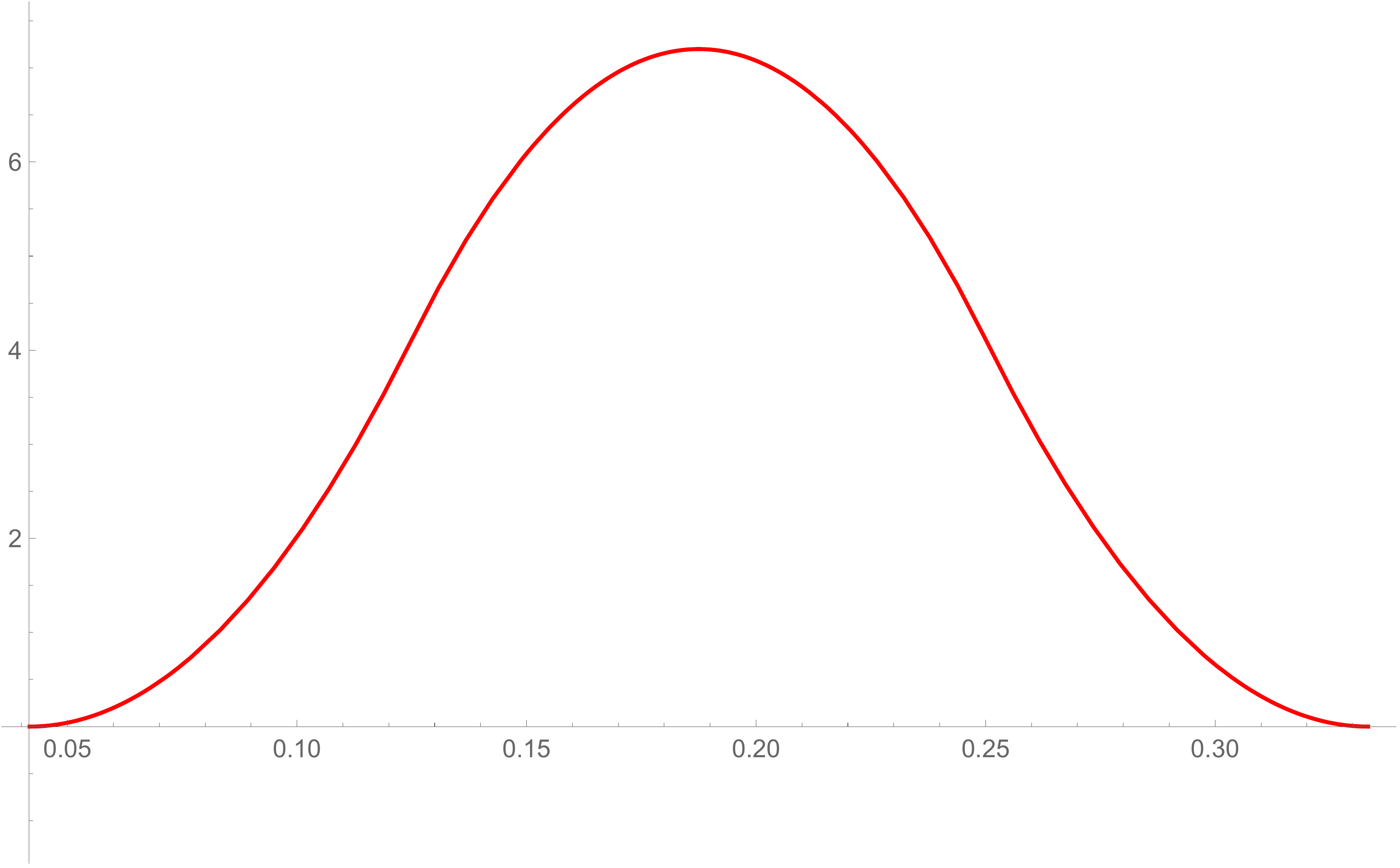}
	                \caption{$S = \<3,4,8,24\>$}
	        \end{subfigure}
\caption{Highly symmetric distribution functions $F(x)$ for two numerical semigroups $S$ chosen by virtue of an Egyptian-fraction identity;
see Example \ref{Example:Egyptian}.}
\label{Figure:Symmetry}
\end{figure}

\begin{Example}\label{Example:Egyptian}
The factorization-length skewness, being expressible in terms of the first and third moments, and the variance,
can be given in closed form:
\begin{equation*}
\Skew \Le\multi{n}
\,\,\sim\,\,
\frac{2 \sqrt{5} (a+b-c-d) (a-b+c-d) (a-b-c+d)}{\left(3 (a^2+b^2+c^2+d^2)   - 2 (a b + a c + b c + a d + b d + c d)\right)^{3/2}},
\end{equation*}
in which $(a,b,c,d) = (\frac{1}{n_1}, \frac{1}{n_2}, \frac{1}{n_3}, \frac{1}{n_4} )$.  In particular, $\Skew \Le \multi{n}$ tends to zero
(that is, $F$ tends to be highly symmetric) if and only if one of the following occurs:
\begin{equation*}
% \frac{1}{n_1} + \frac{1}{n_2} = \frac{1}{n_3} + \frac{1}{n_4},\qquad
\frac{1}{n_1} + \frac{1}{n_3} = \frac{1}{n_2} + \frac{1}{n_4} \qquad \text{or} \qquad
\frac{1}{n_1} + \frac{1}{n_4} = \frac{1}{n_2} + \frac{1}{n_3}.
\end{equation*}
For example, the equalities
\begin{equation*}
\frac{1}{5} + \frac{1}{45} = \frac{1}{6} + \frac{1}{18} 
\qquad \text{and} \qquad
\frac{1}{3} + \frac{1}{24}= \frac{1}{4} + \frac{1}{8} 
\end{equation*}
yield two numerical semigroups with highly symmetric length distribution functions; see Figure \ref{Figure:Symmetry}.
This highlights another connection between Egyptian fractions and the statistical properties of length distributions in numerical semigroups
\cite{lengthdistribution1}.
\end{Example}

Similar computations can be carried out for semigroups with more generators, although it becomes rapidly less rewarding
to search for answers in closed form as the number of generators increases.  We leave the details and particulars
of such computations to the reader.

%%%%%%%%%%%%%%%%%
\subsection{Additional examples}\label{Subsection:More}
In this section, we give two final examples.  
The first points out a curious, but easily explained, phenomenon related to the constant
\begin{equation*}
\delta = \gcd(n_k-n_{k-1}, n_{k-1}-n_{k-2}, \ldots, n_2-n_1),
\end{equation*}
which arises in the semigroup literature as the minimum element of the \emph{delta set} (see~\cite{DSB} for more on this invariant).  Since $n_i \equiv n_j \pmod{\delta}$ for every 
$i, j$, it follows that $\delta$ is the smallest distance that can occur between distinct factorization lengths of $n$, meaning all factorization lengths of a given $n$ are equivalent modulo~$\delta$.  If~$\delta > 1$, then this causes ``gaps'' between positive values in the length multiset.  The question of decomposing $\Le\multi{n}$ along arithmetic sequences is treated in~\cite{lengthdistribution3}.

\begin{Example}
Let $S = \<9,11,13,15,17\>$, for which $\delta = 2$.  If $n$ is even, then every element of $\Le\multi{n}$ is even, and if $n$ is odd, then every element of $\Le\multi{n}$ is odd.    
The corresponding length distribution function is
\begin{equation*}
F(x) = \frac{109395}{32}
\begin{cases}
0 & \text{if $x<\frac{1}{17}$},\\[2pt]
(17 x-1)^3 & \text{if $\frac{1}{17} \leq x < \frac{1}{15}$},\\[2pt]
3 - 129 x + 1833 x^2 - 8587 x^3 & \text{if $\frac{1}{15} \leq x < \frac{1}{13}$},\\[2pt]
-3 + 105 x - 1209 x^2 + 4595 x^3 & \text{if $\frac{1}{13} \leq x < \frac{1}{11}$},\\[2pt]
(1 - 9 x)^3 & \text{if $\frac{1}{11} \leq x < \frac{1}{9}$},\\
0 & \text{if $\frac{1}{9} \leq x$},\\
\end{cases}
\end{equation*}
which appears to be half the height of the upper curve suggested by the blue dots in Figure \ref{Figure:911131517} since the factorization lengths of $n$ all have identical parity.  
In other words, the upper curve suggested by the blue dots in Figure~\ref{Figure:911131517} must be ``averaged out'' by $\delta$ to produce the red curve, which depicts $F(x)$.  The predictions afforded by our methods in this case, as outlined in Table~\ref{Table:Five}, are still surprisingly accurate.
\end{Example}

\begin{figure}
		\begin{subfigure}[b]{0.475\textwidth}
	                \centering
	                \includegraphics[width=\textwidth]{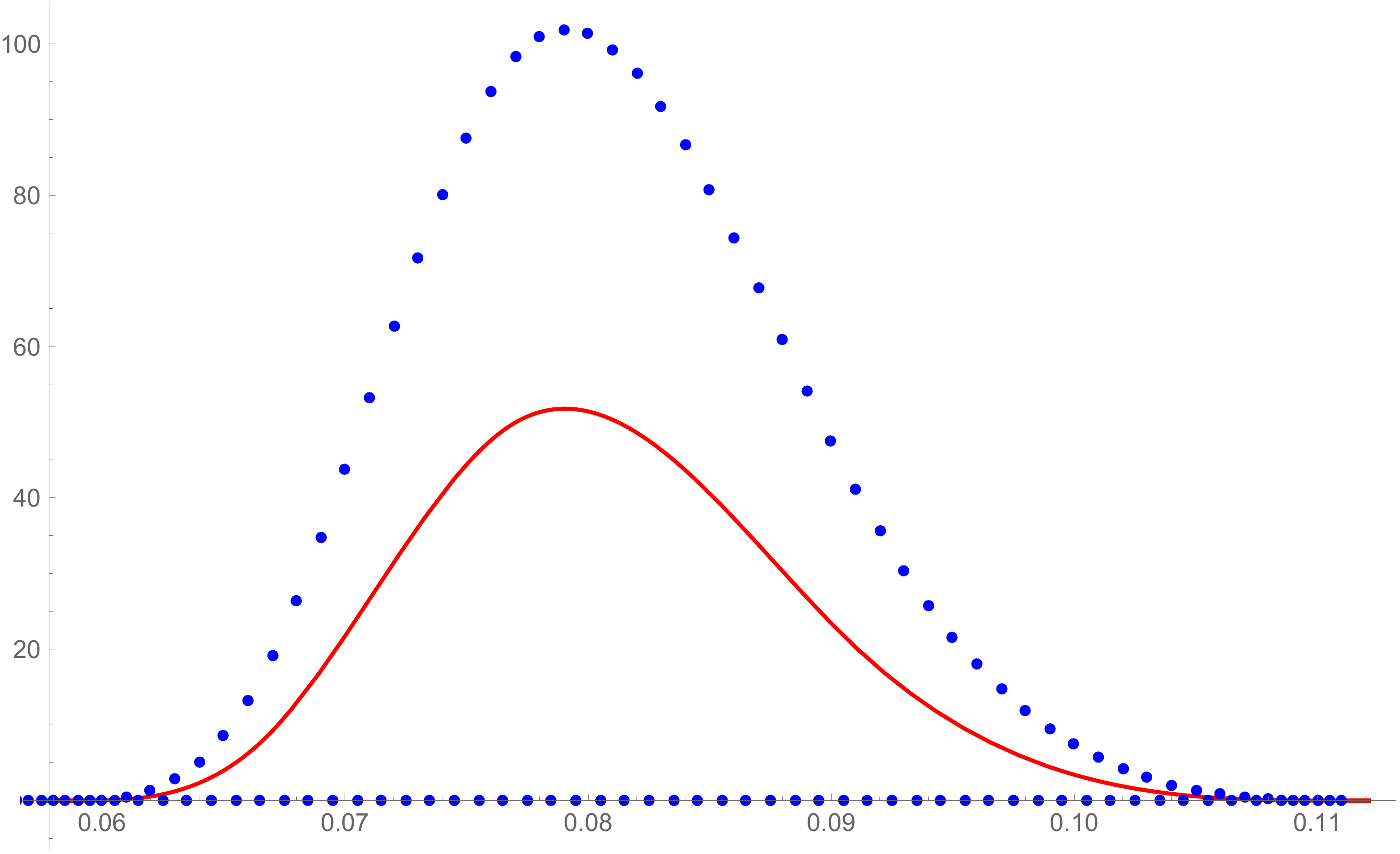}
	                \caption{$n=2{,}000$}
	        \end{subfigure}
		\begin{subfigure}[b]{0.475\textwidth}
	                \centering
	                \includegraphics[width=\textwidth]{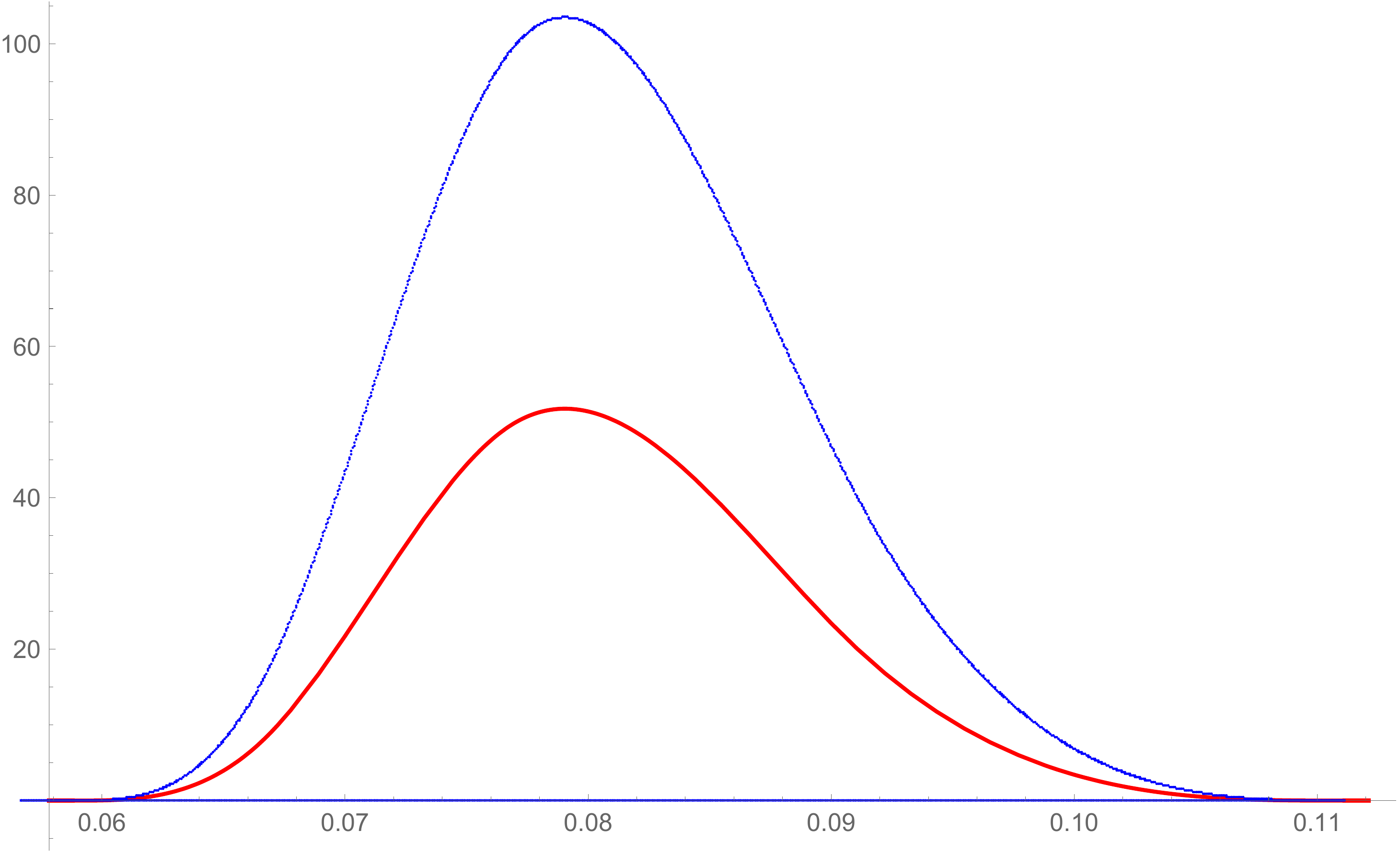}
	                \caption{$n=50{,}000$}
	        \end{subfigure}
\caption{Normalized histogram of the length multiset  $\Le\multi{n}$ (blue) and graph of the length distribution function $F(x)$ (red) for $S =  \<9,11,13,15,17\>$.  For each $i \in \NN$ a blue dot occurs above $i/n$ at height equal to the multiplicity of $i$ in $\Le\multi{n}$.  Since $2{,}000$ and $50{,}000$ are even, there are 
no factorizations of odd length.  Thus, the red curve is half the height of the upper curve suggested by the blue points.}
\label{Figure:911131517}
\end{figure}

\begin{table}\footnotesize
\begin{equation*}
\begin{array}{c|cc||c|cc}
\text{Statistic} & \text{Actual} & \text{Predicted} &\text{Statistic} & \text{Actual} & \text{Predicted} \\
\hline
\Mean \Le\multi{10^5} & 8088.80 & 8088.67 &\HarMean \Le\multi{10^5} & 8019.043  & 8018.96\\[3pt]
\Median \Le\multi{10^5} & 8038& 8037.53& \GeoMean \Le\multi{10^5} & 8053.75 &8053.64 \\[3pt]
\Mode\Le\multi{10^5} &7904& 7904.25 & \Skew \Le\multi{10^5} & 0.32812710 & 0.32812712 \\[3pt]
\StDev\Le\multi{10^5} &757.14  &  756.89& \min / \max \Le \multi{10^5} &\scalebox{0.85}{11110/5884} & \scalebox{0.85}{11111.11/5882.35}\\[3pt]
\end{array}
\end{equation*}
\caption{Actual versus predicted statistics (rounded to two decimal places) 
for $\Le\multi{10^5}$, the multiset of factorization lengths of $100{,}000$, in 
$S = \<9,11,13,15,17\>$.}
\label{Table:Five}
\end{table}

We conclude with one final example that demonstrates the impressive estimates our techniques afford for a numerical semigroup with nine generators.

\begin{Example}
Consider $S= \<118,150,162,175,182,258,373,387,456\>$, which has nine generators.
Describing the length-distribution statistics of such a semigroup is far beyond the realm of previously-established techniques.
We spare the reader the display of the explicit length-generating function; suffice it to say, $F$ is a
piecewise polynomial function of degree $7$ (see Figure~\ref{Figure:Nine}).
A few computations with a computer algebra system provide accurate approximations to the relevant statistics; see Table~\ref{Table:Nine}.
\end{Example}

\begin{figure}
\centering
		\begin{subfigure}[b]{0.475\textwidth}
	                \centering
	                \includegraphics[width=\textwidth]{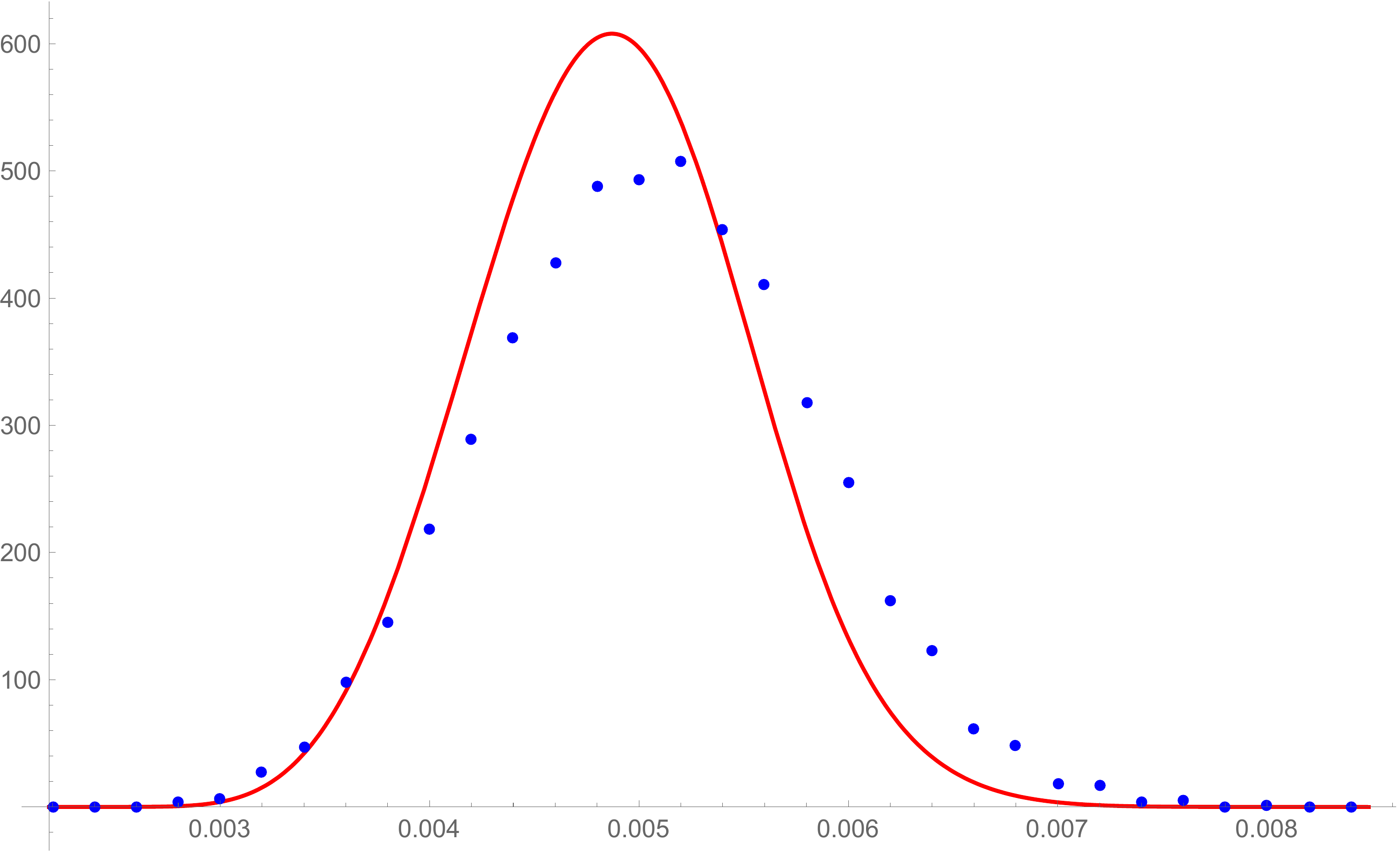}
	                \caption{$n=5{,}000$}
	        \end{subfigure}
		\begin{subfigure}[b]{0.475\textwidth}
	                \centering
	                \includegraphics[width=\textwidth]{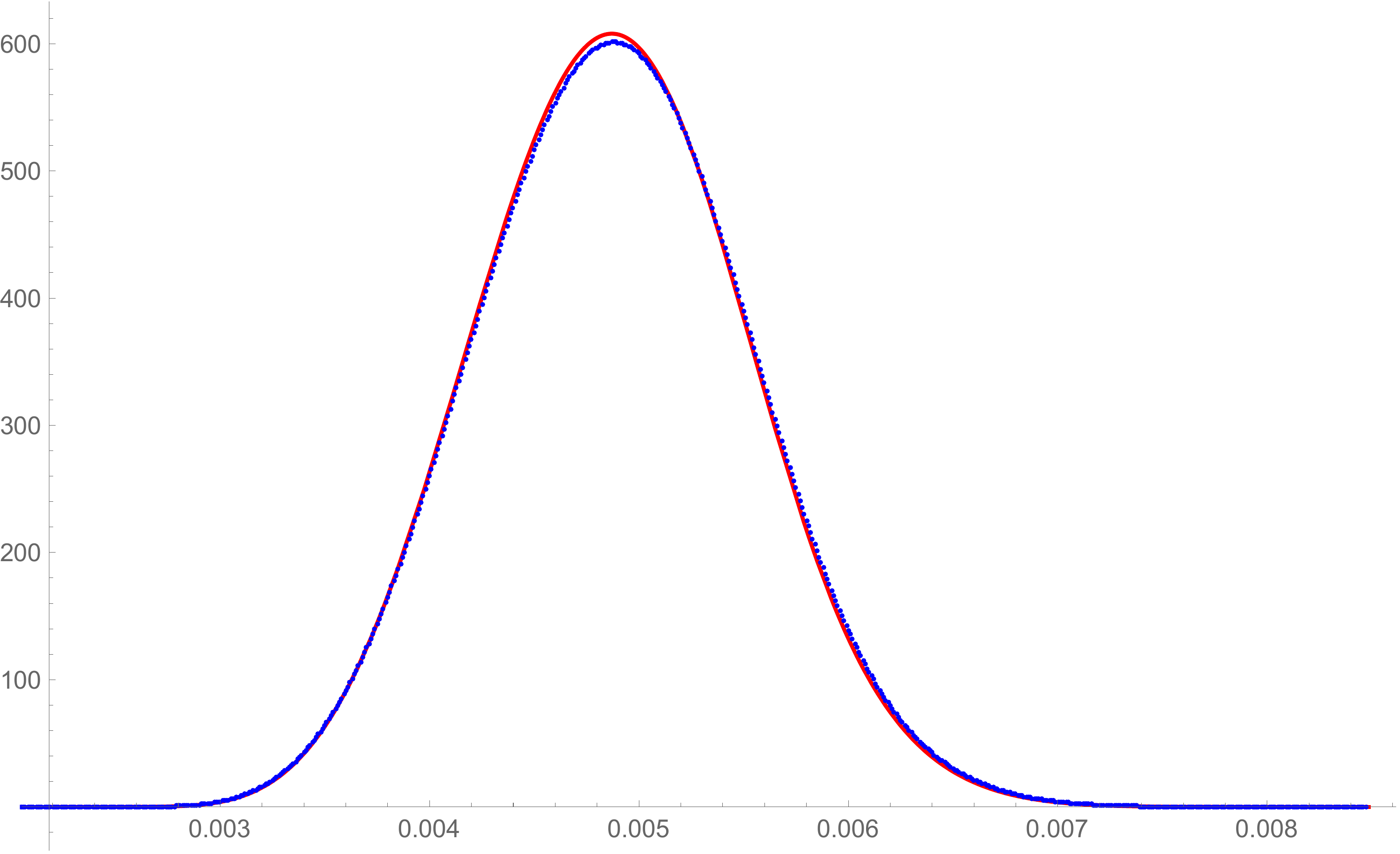}
	                \caption{$n=100{,}000$}
	        \end{subfigure}

\caption{Normalized histogram of the length multiset  $\Le\multi{n}$ (blue) and graph of the length distribution function $F(x)$ (red) for the numerical semigroup
$\<118,150,162,175,182,258,373,387,456\>$.  For each $i \in \NN$ a blue dot occurs above $i/n$ at height equal to the multiplicity of $i$ in $\Le\multi{n}$.}
\label{Figure:Nine}
\end{figure}

\begin{table}\footnotesize
\begin{equation*}
\begin{array}{c|cc||c|cc}
\text{Statistic} & \text{Actual} & \text{Predicted} &\text{Statistic} & \text{Actual} & \text{Predicted} \\
\hline
\Mean \Le\multi{10^5} & 488.30 & 487.30 &\HarMean \Le\multi{10^5} & 479.46  & 478.64\\[3pt]
\Median \Le\multi{10^5} &488 & 486.74& \GeoMean \Le\multi{10^5} &  483.92& 483.00\\[3pt]
\Mode\Le\multi{10^5} &488& 487.08 & \Skew \Le\multi{10^5} & 0.09692 & 0.09699 \\[3pt]
\StDev\Le\multi{10^5} &  65.01& 64.29 & \min / \max \Le \multi{10^5} &\scalebox{0.9}{221/846} & \scalebox{0.9}{219.30/847.46}\\[3pt]
\end{array}
\end{equation*}
\caption{Actual versus predicted statistics (rounded to two decimal places) 
for $\Le\multi{10^5}$ in $S = \<118,150,162,175,182,258,373,387,456\>$.}
\label{Table:Nine}
\end{table}

%%%%%%%%%%%%%%%%%%%%%%%%%%%%%%%%%%%%%%%%%%%%%
\section{Complete homogeneous symmetric polynomials}\label{Section:CHS}
There are several unexpected consequences of our work 
to the realm of symmetric functions.  The next theorem provides a probabilistic interpretation
of the complete homogeneous symmetric polynomials and a means to extend their definition
to nonintegral degrees.  From this result we recover a well-known positivity result (Corollary \ref{Corollary:Hunter}).  

The proof of Theorem~\ref{Theorem:CHS} was recently refined~\cite{Bottcher} using the theory of splines.  

\begin{Theorem}\label{Theorem:CHS}
Let $x_1 < x_2< \ldots <x_k$ be real numbers.  Then
\begin{equation}\label{eq:Faaa}
H(x;x_1,x_2,\ldots,x_k) = \frac{k-1}{2} \sum_{r=1}^k \frac{|x_r-x|(x_r-x)^{k-3}}{\prod_{j \neq r}(x_r-x_j)} 
\end{equation}
is a probability distribution on $\RR$ with support $[x_1,x_k]$.  Moreover, for $p\in \NN$,
\begin{equation}\label{eq:Analytic}
h_p (x_1,x_2,\ldots,x_k) = \binom{p+k-1}{p}\int_{x_1}^{x_k} t^p  H(t;x_1,x_2,\ldots,x_k) \, dt .
\end{equation}
\end{Theorem}

\begin{proof}
By continuity, we may assume $x_1,x_2,\ldots,x_k \in \QQ$.  In fact, it is not hard to see that we can further assume
\begin{equation*}
x_1 = \tau + \frac{m}{n_k}, \qquad x_2 = \tau  + \frac{m}{n_{k-1}},\ldots, \qquad x_k = \tau+\frac{m}{n_1},
\end{equation*}
in which $\tau \in \ZZ$, $m \in \NN$, and $n_1,n_2,\ldots,n_k \in \NN$ satisfy
$\gcd(n_1,n_2,\ldots,n_k)=1$ and $n_1<n_2<\ldots<n_k$.
Observe that
\begin{align*}
H(x;x_1,x_2,\ldots,x_k)
&=H(x-\tau;x_1-\tau,x_2-\tau,\ldots,x_k-\tau) \\
&= \frac{1}{m} H\bigg( \frac{x-\tau}{m}; \frac{x_1 - \tau}{m}, \frac{x_2- \tau}{m},\ldots, \frac{x_k-\tau}{m} \bigg) \\
&= \frac{1}{m} H\bigg( \frac{x-\tau}{m};  \frac{1}{n_1}, \frac{1}{n_2},\ldots, \frac{1}{n_k} \bigg) \\
&= \frac{1}{m} F\bigg( \frac{x-\tau}{m} \bigg),
\end{align*}
in which $F$ denotes the function from Theorem \ref{Theorem:Main}a; to see this compare \eqref{eq:MidT} with \eqref{eq:Faaa}.
Since $F$ is a probability distribution on $\RR$ supported on $[\frac{1}{n_k},\frac{1}{n_1}]$, we conclude that
$H(x;x_1,x_2,\ldots,x_k)$ is a probability distribution on $\RR$ supported on $[x_1,x_k]$.  
Let $\nu$ denote the corresponding probability
measure, which satisfies $\nu(A) = \int_A H(t;x_1,x_2,\ldots,x_k)\,dt$ for all Borel sets $A \subseteq \RR$, and let 
\begin{equation}\label{eq:Compare1}
\phi_{\nu}(z) = \int_{x_1}^{x_k}  e^{itz}\, d\nu(t) = \sum_{p=0}^{\infty} m_p(\nu) \frac{(iz)^p}{p!}
\end{equation}
denote the corresponding characteristic function.    Then,
\begin{align}
\phi_{\nu}(z) 
&= \big(H(x;x_1,x_2,\ldots,x_k)\big)^{\!\!\widehat{\,\,}}(z) \nonumber\\
&= (k-1)!\sum_{r=1}^k \frac{e^{ix_r z}}{(iz)^{k-1}\prod_{j\neq r}(x_r-x_j)}  && (\text{by \eqref{eq:Backward}}) \nonumber\\
&= (k-1)!\sum_{p=0}^{\infty} \frac{h_p(x_1,x_2,\ldots,x_k)}{(p+k-1)!} (iz)^p && (\text{by Theorem \ref{Theorem:Combo}}) \nonumber\\
&=\sum_{p=0}^{\infty} \binom{p+k-1}{p}^{-1} h_p(x_1,x_2,\ldots,x_k)\frac{(iz)^p}{p!} .\label{eq:Compare2}
\end{align}
For $p\in \NN$, compare \eqref{eq:Compare1} and \eqref{eq:Compare2} and obtain \eqref{eq:Analytic}.
\end{proof}

As an immediate corollary, we obtain a short proof of the positive-definiteness of complete homogeneous symmetric functions of even degree.
This dates back to D.B.~Hunter \cite{Hunter} (a somewhat stronger version was recently obtain by T.~Tao \cite{Tao}).

\begin{Corollary}\label{Corollary:Hunter}
If $x_1,x_2,\ldots,x_k \in \RR \backslash\{0\}$, then $h_{2d}(x_1,x_2,\ldots,x_k) \geq 0$.
\end{Corollary}

\begin{proof}
By symmetry and continuity, we may assume $x_1 < x_2 <  \cdots < x_k$.  Then
\begin{equation*}
h_{2d}(x_1,x_2,\ldots,x_k)
= \binom{2d+k-1}{2d} \int_{x_1}^{x_k} t^{2d} H(t;x_1,x_2,\ldots,x_k)\,dt > 0. \qedhere
\end{equation*}
\end{proof}

In light of Theorem \ref{Theorem:CHS}, we can define
\begin{equation}\label{eq:Z}
h_z(x_1,x_2,\ldots,x_k) := \frac{(z+k-1)\cdots(z+1)}{(k-1)!}
\int_{\RR} t^z H(t;x_1,x_2,\ldots,x_k)\,dt 
\end{equation}
for nonintegral $z$.  Since $H(t;x_1,x_2,\ldots,x_k)$ is a symmetric function of the variables $x_1,x_2,\ldots,x_k$,
it follows that $h_z(x_1,x_2,\ldots,x_k)$ is a symmetric function.  
Moreover, $H(t;x_1,x_2,\ldots,x_k)$ is piecewise polynomial and hence the right-hand side of \eqref{eq:Z} is explicitly computable.
Thus, \eqref{eq:Z} provides a natural notion of complete homogeneous symmetric polynomials of arbitrary degree.
This complements recent work of T.~Tao, who developed a notion of symmetric functions in a fractional number of variables \cite{Tao2}.
Tao's approach, inspired by work of Bennett--Carbery--Tao on the multilinear restriction 
and Kakeya conjectures from harmonic analysis \cite{Bennett},
is also based upon a probabilistic framework.

\begin{Example}
For $k=3$ and distinct $a,b,c \in \RR$, we have
\begin{equation*}
H(x; a,b,c) = \frac{|a-x|}{(b-a)(c-a)} + \frac{|b-x|}{(a-b)(c-b)} + \frac{|c-x|}{(a-c)(b-c)}
\end{equation*}
and
\begin{align*}
h_z(a,b,c)
%&=\frac{(z+2)(z+1)}{2} \int_{\RR} t^z H(t;a,b,c)\,dt  \\
&=\frac{a^{z+2}(b-c) + b^{z+2}(c-a) + c^{z+2}(a-b) }{(a-b) (a-c) (b-c)}.
\end{align*}
As expected, if we apply \eqref{eq:Z} for $z \in \NN$ we obtain
the complete homogeneous symmetric polynomials
\begin{equation*}
h_0(a,b,c)=1,\quad
h_1(a,b,c) = a+b+c,\quad
h_2(a,b,c) = a^2 + b^2+c^2 + a b + bc + ca,
\end{equation*}
and so forth.
For $a,b,c > 0$, we obtain curious symmetric functions such as
\begin{align*}
h_{\frac{1}{2}}(a,b,c) 
&= \frac{ a^{\frac{5}{2}}(b-c)   +b^{\frac{5}{2}}(c-a) + c^{\frac{5}{2}}(a-b) }{(a-b) (a-c) (b-c)} ,\\
h_{-\frac{1}{2}}(a,b,c) 
&= \frac{\sqrt{a} \sqrt{b}+\sqrt{a} \sqrt{c}+\sqrt{b} \sqrt{c}}{(\sqrt{a}+\sqrt{b}) (\sqrt{a}+\sqrt{c}) (\sqrt{b}+\sqrt{c})} ,\\
h_{-\frac{3}{2}}(a,b,c)
&=-\frac{1}{(\sqrt{a}+\sqrt{b}) (\sqrt{a}+\sqrt{c}) (\sqrt{b}+\sqrt{c})} ,\\
h_{-\frac{5}{2}}(a,b,c)
&= \frac{\frac{a-b}{\sqrt{c}}+\frac{b-c}{\sqrt{a}}+\frac{c-a}{\sqrt{b}}}{(a-b) (a-c) (b-c)} ,\\
h_{-3}(a,b,c)
&=\frac{1}{abc} ,\\
h_{-4}(a,b,c)
&= \frac{a b+a c+b c}{a^2 b^2 c^2},
\end{align*}
We also note that $h_{-1}(a,b,c) = h_{-2}(a,b,c) = 0$.
\end{Example}
For negative and positive rational values of $z$, one can explicitly describe $h_z(x_1,x_2,\ldots,x_k)$ in terms of Schur functions, 
though we do not wish to be drawn too far afield here.  We intend to take this subject up in a subsequent publication.

%%%%%%%%%%%%%%%%%%%%%%%%%%%%%%%%%%%%%%%%%%%%%%%%%%%
\section{Proof of Theorem \ref{Theorem:Moment}}\label{Section:ProofMoment}
The first part of the proof concerns a certain two-variable generating function (Subsection \ref{Subsection:Generating}).  Next comes a lengthy residue computation
(Subsection \ref{Subsection:Residue}).  A~few power series computations complete the proof (Subsection \ref{Subsection:Complete}).

\subsection{Generating function}\label{Subsection:Generating}
Fix $S = \<n_1,n_2,\ldots,n_k\>$ with $\gcd(n_1,n_2,\ldots,n_k)=1$  and consider the generating function
\begin{align}
g(z,w)
& := \prod_{i=1}^k \frac{1}{1 - wz^{n_i}}\label{eq:Gzw} \\
&= \prod_{i=1}^k (1 + wz^{n_i} + w^2z^{2n_i} + \cdots) \nonumber \\
&= \sum_{a_1,a_2,\ldots,a_k \geq 0} w^{a_1 + a_2 + \cdots + a_k}\, z^{a_1 n_1 + a_2 n_2 + \cdots + a_k n_k} \nonumber \\
&= \sum_{n=0}^\infty z^n \sum_{\ell=0}^\infty (\text{\# of factorizations of $n$ of length $\ell$}) w^{\ell} \nonumber \\  
&= \sum_{n=0}^\infty z^n \sum_{\ell \in \Le\multi{n}} w^{\ell}. \nonumber
\end{align}
Then
\begin{equation*}
\left( \! w \frac{\partial}{\partial w} \! \right)^{\!\!p} \!\! g(z,w) 
=\sum_{n=0}^\infty z^n \sum_{\ell \in \Le[n]}  w^{\ell} \ell^p
\end{equation*}
and hence 
\begin{equation*}
\Lambda_p(n):=\sum_{\ell \in \Le\multi{n}} \ell^p
\end{equation*}
is the coefficient of $z^n$ in the series expansion of
\begin{equation}\label{eq:Gdef}
G(z) := \bigg( \! w \frac{\partial}{\partial w} \! \bigg)^p \!\! g(z,w) \Bigg |_{w=1}.
\end{equation}
To make use of this we require the following lemma.

\begin{Lemma}
For $p\in \NN$,
\begin{equation}\label{eq:partialidentity}
\frac{\partial^p}{\partial w^p} g(z,w)
= p! \left( \prod_{b=1}^k \frac{1}{1 - wz^{n_b}} \right) h_p \left( \frac{z^{n_1}}{1 - wz^{n_1}}, \cdots, \frac{z^{n_k}}{1 - wz^{n_k}} \right).
\end{equation}
\end{Lemma}

\begin{proof}
We proceed by induction.
The base case $p=0$ is \eqref{eq:Gzw}.
For the inductive step, suppose that \eqref{eq:partialidentity} holds for some $p \in \NN$.  Then 
\begin{align*}
\frac{1}{p!} \frac{\partial^{p+1}}{\partial w^{p+1}} g(z,w) 
&= \frac{\partial}{\partial w} \left( \frac{1}{p!}\frac{\partial^{p}}{\partial w^{p}}  g(z,w) \right)  \\
&= \frac{\partial}{\partial w} \left( \prod_{b=1}^k \frac{1}{1 - wz^{n_b}} \right) h_p \left( \frac{z^{n_1}}{1 - wz^{n_1}}, \cdots, \frac{z^{n_k}}{1 - wz^{n_k}} \right) \\
&= \left( \prod_{b=1}^k \frac{1}{1 - wz^{n_b}} \right)\left(\sum_{i = 1}^k  \frac{z^{n_i}}{1 - wz^{n_i}}   \right)
h_p \left( \frac{z^{n_1}}{1 - wz^{n_1}}, \cdots, \frac{z^{n_k}}{1 - wz^{n_k}} \right)  \\
&\qquad\qquad+ \left( \prod_{b=1}^k \frac{1}{1 - wz^{n_b}} \right) \frac{\partial}{\partial w} h_p \left( \frac{z^{n_1}}{1 - wz^{n_1}}, \cdots, \frac{z^{n_k}}{1 - wz^{n_k}} \right) \\
&= (p+1) \left( \prod_{b=1}^k \frac{1}{1 - wz^{n_b}} \right) h_{p+1} \left( \frac{z^{n_1}}{1 - wz^{n_1}}, \cdots, \frac{z^{n_k}}{1 - wz^{n_k}} \right).
\end{align*}
The final equality follows from counting how many times each term appears when the 
line before it is expanded.
\end{proof}

The \emph{Stirling number of the second kind}, denoted $\stirling{n}{i}$, counts the number of partitions of 
$[n] := \{1, 2, \ldots, n\}$ into $i$ nonempty subsets. These numbers satisfy
\begin{equation*}
\stirling{n+1}{i} = i \stirling{n}{i} + \stirling{n}{i-1}
\end{equation*}
and
\begin{equation}\label{eq:stirlingidentity}
\bigg(x \frac{d}{dx}\bigg)^{\!p}
= \sum_{i=0}^p \stirling{p}{i} x^i \frac{d^i}{dx^i},
\end{equation}
which holds for $p \in \NN$ \cite{Carlitz1, Carlitz2, Toscano}.

From \eqref{eq:Gdef}, and then \eqref{eq:stirlingidentity} and \eqref{eq:partialidentity}, we obtain
\begin{align}
G(z)
&=  \bigg( \! w \frac{\partial}{\partial w} \! \bigg)^p \!\! g(z,w)  \Bigg |_{w=1} \nonumber\\
&= \sum_{i=0}^p \stirling{p}{i} \, w^i \frac{\partial^i g}{\partial w^i} \, \Bigg |_{w=1} \nonumber \\
&= \sum_{i=0}^p \stirling{p}{i} \, i! \, w^i \bigg( \prod_{j=1}^k \frac{1}{1 - wz^{n_j}} \bigg) h_i \bigg( \frac{z^{n_1}}{1 - wz^{n_1}}, \cdots, \frac{z^{n_k}}{1 - wz^{n_k}} \bigg) \Bigg |_{w=1} \nonumber \\
&= \sum_{i=0}^p \stirling{p}{i} \, i!\, \left( \prod_{j=1}^k \frac{1}{1 - z^{n_j}}  \right) h_i \bigg( \frac{z^{n_1}}{1 - z^{n_1}}, \cdots, \frac{z^{n_k}}{1 - z^{n_k}} \bigg).\label{eq:FH}
\end{align}
Thus, $G(z)$ is a rational function in $z$, all of whose poles are certain $L$th roots of unity, in which 
\begin{equation*}
L:=\lcm(n_1,n_2,\ldots,n_k).
\end{equation*}
Each $1 - z^{n_i}$ factors as a product of $n_i$ distinct linear factors, one of which is $1 - z$.
Consequently, $1$ is a pole of $G(z)$ of order $k + p$; this arises from the summand corresponding to 
$i=p$.  Moreover, $k+p$ 
is the maximum possible order for a pole of $G(z)$, and $1$ is the unique pole of this order.
Indeed, $\gcd(n_1, n_2,\ldots,n_k) = 1$ ensures that the only common root of
$1-z^{n_1}, 1-z^{n_2},\ldots, 1-z^{n_k}$ is $1$.

Thus, $\Lambda_p(n)$ is a complex linear combination of terms of the form
\begin{equation*}
n^{r-1} \omega^n,\,\,
n^{r-2} \omega^n,\ldots,\,\,
\omega^n,
\end{equation*}
in which $\omega$ is a pole of $G(z)$ of order at most $r$ (see \cite[Ch.~1]{continuousdiscretely} for an overview of this method).  The unique pole of $G(z)$ of highest order is $1$, which has order $k+p$.
Thus, there exist periodic functions $a_0, a_1, \ldots, a_{k+p-1}: \NN \to \CC$ with periods dividing $L$ such that 
\begin{equation*}
\Lambda_p(n) = a_{k+p-1}(n) n^{k+p-1} + a_{k+p-2}(n)n^{k+p-2} + \cdots + a_1(n)n+a_0(n).
\end{equation*}
This establishes the desired quasipolynomial representation.  It remains to show 
\begin{equation*}
a_{k+p-1}(n) = \frac{p!}{(k + p - 1)! (n_1 n_2 \cdots n_k)} h_p \bigg( \frac{1}{n_1}, \frac{1}{n_2}, \ldots, \frac{1}{n_k} \bigg) 
\end{equation*}
and that the periodic functions $a_i(n)$ assume only rational values.

\subsection{A residue computation}\label{Subsection:Residue}
Since $G(z)$ has a pole of order $k+p$ at $1$, we have
\begin{equation}\label{eq:Fcg}
G(z) = \frac{C}{(1 - z)^{k+p}} + u(z)
\end{equation}
for some constant $C$ and some rational function $u$, all of whose poles are
$L$th roots of unity with order at most $k + p - 1$.  In particular,
\begin{equation*}
u(z) = \sum_{n=0}^\infty w(n) z^n
\end{equation*}
for some quasipolynomial $w(n)$ of degree at most $k + p - 2$ and period dividing~$L$. 

The only summand in \eqref{eq:FH} that has a pole at $1$ of order $k+p$ is the
term that corresponds to $i=p$.  
The summands in \eqref{eq:FH} with $0 \leq i \leq p-1$ satisfy
\begin{equation*}
\lim_{z \to 1} (1 - z)^{k+p}
\underbrace{
\left( \prod_{j=1}^k \frac{1}{1 - z^{n_j}}  \right) h_i \bigg( \frac{z^{n_1}}{1 - z^{n_1}}, \cdots, \frac{z^{n_k}}{1 - z^{n_k}} \bigg)
}_{\text{has a pole at $z=1$ of order $k+i \leq k+p-1$}}
= 0.
\end{equation*}
Consequently,
\begin{align}
C
&= \lim_{z \to 1} (1-z)^{k+p} G(z) \nonumber \\
&= \lim_{z \to 1} \sum_{i=0}^p \stirling{p}{i} \, i! \, (1 - z)^{k+p} \left( \prod_{j=1}^k \frac{1}{1 - z^{n_j}} \right) h_i \bigg( \frac{z^{n_1}}{1 - z^{n_1}}, \cdots, \frac{z^{n_k}}{1 - z^{n_k}} \bigg) \nonumber \\
&= p!\lim_{z \to 1} (1 - z)^{k+p} \left( \prod_{j=1}^k \frac{1}{1 - z^{n_j}} \right) h_p \bigg( \frac{z^{n_1}}{1 - z^{n_1}}, \cdots, \frac{z^{n_k}}{1 - z^{n_k}} \bigg) \nonumber \\
&= p! \lim_{z \to 1} \left( \prod_{j=1}^k \frac{1}{1 + \cdots + z^{n_j-1}} \right) h_p \bigg( \frac{z^{n_1}}{1 + \cdots + z^{n_1-1}}, \cdots, \frac{z^{n_k}}{1 + \cdots + z^{n_k-1}} \bigg) \nonumber \\
&= \frac{p!}{n_1 n_2 \cdots n_k} h_p \bigg( \frac{1}{n_1}, \frac{1}{n_2}, \ldots, \frac{1}{n_k} \bigg).
\label{eq:c}
\end{align}

\subsection{Completing the proof}\label{Subsection:Complete}
Observe that
\begin{align*}
\frac{1}{(1 - z)^{k+p}}
&= \sum_{n=0}^\infty \binom{n + k + p - 1}{k + p - 1} z^n \\
&= \sum_{n=0}^\infty \frac{(n + k + p - 1) \cdots (n + 1)}{(k + p - 1)!} z^n \\ 
&= \frac{1}{(k + p - 1)!} \sum_{n=0}^\infty \big(n^{k+p-1} + v(n)\big) z^n,
\end{align*}
in which $v(n)$ is a quasipolynomial of degree $k + p - 2$ with integer coefficients.
Together with~\eqref{eq:Fcg} and~\eqref{eq:c}, we obtain
\begin{align*}
G(z)
&= \frac{p!}{n_1 \cdots n_k} h_p \bigg( \frac{1}{n_1}, \frac{1}{n_2}, \ldots, \frac{1}{n_k} \bigg) \cdot \frac{1}{(1 - z)^{k+p}} + u(z) \\ 
&= \frac{p!}{(k + p - 1)! n_1 \cdots n_k} h_p \bigg( \frac{1}{n_1}, \ldots, \frac{1}{n_k} \bigg) \sum_{n=0}^\infty (n^{k+p-1} + v(n)) z^n + \sum_{n=0}^\infty w(n)z^n.
\end{align*}
Thus,
\begin{equation*}
\Lambda_p (n) = \frac{p!}{(k + p - 1)! n_1 n_2 \cdots n_k} h_p \bigg( \frac{1}{n_1}, \frac{1}{n_2}, \ldots, \frac{1}{n_k} \bigg) n^{k+p-1} + q(n),
\end{equation*}
in which $q(n)$ is a quasipolynomial of degree at most $k+p-2$ whose coefficients have periods dividing $L$.  Additionally, since $v(n)$ and $w(n)$ both have rational coefficients, so must $q(n)$.  This completes the proof.  \qed

%%%%%%%%%%%%%%%%%%%%%%%%%%%%
\section{Proof of Theorem \ref{Theorem:Combo}}\label{Section:Combo}
We wish to prove the exponential generating function identity
\begin{equation}\label{eq:Mo}
\sum_{p=0}^{\infty} \frac{h_p(x_1,x_2,\ldots,x_k)}{(p+k-1)!} z^{p+k-1}
= \sum_{r=1}^k \frac{e^{x_r z}}{\prod_{j\neq r}(x_r-x_j)},
\end{equation}
valid for $z \in \CC$.  We first show that the power series on the left-hand side of \eqref{eq:Mo}
has an infinite radius of convergence (Subsection \ref{Subsection:RoC}).
Then we reduce \eqref{eq:Mo} to an identity that links 
complete homogeneous symmetric polynomials to the determinants of certain
Vandermonde-like matrices (Subsection \ref{Subsection:Vandermonde}).
A brief excursion into algebraic combinatorics (Subsection \ref{Subsection:AlgCom})
finishes off the proof.

\subsection{Radius of convergence}\label{Subsection:RoC}
Fix distinct $x_1,x_2,\ldots,x_k \in \CC \backslash\{0\}$.
We claim that the radius of convergence of the power series
\begin{equation}\label{eq:RoCs}
\sum_{p=0}^{\infty} \frac{h_p(x_1,x_2,\ldots,x_k)}{(p+k-1)!} z^{p+k-1}
\end{equation}
is infinite.  This ensures that \eqref{eq:Mo} is an equality of entire functions.
The ordinary generating function for the complete homogeneous symmetric polynomials is
\begin{equation}\label{eq:CHSPgenerating}
\sum_{p=0}^\infty h_p(x_1, x_2, \ldots, x_k)z^p
=\prod_{i=1}^k \frac{1}{1 - x_i z};
\end{equation}
see~\cite{ec1}.   
The radius of convergence of the preceding power series is
the distance from $0$ to the closest pole $1/x_1,1/x_2,\ldots,1/x_k$.
Consequently, the Cauchy--Hadamard formula \cite[p.~55]{Sarason} yields
\begin{equation*}
 \limsup_{p\to\infty} |h_p(x_1, x_2, \ldots, x_k)|^{\frac{1}{p}} = \operatorname{max}\{ |x_1|,|x_2|,\ldots,|x_k|\}.
\end{equation*}
Since
\begin{equation*}
\lim_{p\to\infty} \big( (p+k-1)! \big)^{\frac{1}{p}}
\geq \lim_{p\to\infty} ( p!)^{\frac{1}{p}}
\geq \lim_{p\to\infty} \Big(\big( \tfrac{p}{3}\big)^{\frac{p}{3}}\Big)^{\frac{1}{p}} 
= \lim_{p\to\infty} \left(\frac{p}{3}\right)^{\frac{1}{3}}
=\infty,
\end{equation*}
a second appeal to the Cauchy--Hadamard formula tells us that the  radius of convergence $R$ of \eqref{eq:RoCs}
satisfies
\begin{equation*}
\frac{1}{R} 
= \limsup_{p\to\infty} \left( \frac{h_p(x_1,x_2,\ldots,x_k)}{(p+k-1)!} \right)^{\frac{1}{p}}
= \frac{\operatorname{max}\{ |x_1|,|x_2|,\ldots,|x_k|\}}{\displaystyle\lim_{p\to\infty} ( (p+k-1)!)^{\frac{1}{p}}} = 0.
\end{equation*}
Thus, the radius of convergence of \eqref{eq:RoCs} is infinite.

\subsection{A Vandermonde-like determinant}\label{Subsection:Vandermonde}
The determinant of the $k\times k$ \emph{Vandermonde matrix}
\begin{equation*}
V(x_1,x_2,\ldots,x_k) 
:= 
\begin{bmatrix} 
1 & x_1 & x_1^2 & \cdots & x_1^{k-1} \\[3pt] 
1 & x_2 & x_2^2 & \cdots & x_2^{k-1} \\[3pt] 
\vdots & \vdots & \vdots & \ddots & \vdots \\[3pt]
1 & x_k & x_k^2 & \cdots & x_k^{k-1} 
\end{bmatrix}
\end{equation*}
is 
\begin{equation*}
\det V(x_1,x_2,\ldots,x_k) = \prod_{1 \leq i<j \leq n} (x_j-x_i);
\end{equation*}
see \cite[p.~37]{HJ}.  
In what follows, $V(x_1,\ldots,\widehat{x_r},\ldots,x_k)$
denotes the $(k-1)\times (k-1)$ Vandermonde matrix obtained from $V(x_1,x_2,\ldots,x_k)$ by removing 
$x_r$ (do not confuse the carat with the Fourier transform).  Cofactor expansion and 
the linearity of the determinant in the final column of a matrix reveals that
\begin{align*}
&\det V(x_1,x_2,\ldots,x_k)\sum_{r=1}^k \frac{e^{x_r z}}{\prod_{j\neq r}(x_r-x_j)}  \\
&\qquad = \sum_{r=1}^k (-1)^{k-r} \dfrac{\det V(x_1,x_2,\ldots,x_k)}{\big(\prod_{j<r}(x_r-x_j) \big)\big(\prod_{j>r}(x_j-x_r)\big)}e^{x_rz} \\
&\qquad = \sum_{r=1}^k (-1)^{k-r} \det V(x_1,\ldots,\widehat{x_r},\ldots,x_k)
e^{x_rz} \\
&\qquad = \det \small
\begin{bmatrix} 
1 & x_1 & x_1^2 & \cdots & x_1^{k-2} & e^{x_1z} \\[3pt] 
1 & x_2 & x_2^2 & \cdots & x_2^{k-2} & e^{x_2z} \\[3pt] 
\vdots & \vdots & \vdots & \ddots & \vdots & \vdots \\[3pt]
1 & x_k & x_k^2 & \cdots & x_k^{k-2} & e^{x_kz} \\  
\end{bmatrix} \\
&\qquad= \sum_{i=0}^{\infty} \frac{z^i}{i!}\det\small
\begin{bmatrix} 
1 & x_1 & x_1^2 & \cdots & x_1^{k-2} & x_1^i  \\[3pt] 
1 & x_2 & x_2^2 & \cdots & x_2^{k-2} & x_2^i  \\[3pt] 
\vdots & \vdots & \vdots & \ddots & \vdots & \vdots \\[3pt]
1 & x_k & x_k^2 & \cdots & x_k^{k-2} & x_k^i  \\  
\end{bmatrix}
\\
&\qquad= \sum_{p=0}^{\infty}\frac{z^{p+k-1}}{(p+k-1)!} \det\small
\begin{bmatrix} 
1 & x_1 & x_1^2 & \cdots & x_1^{k-2} & x_1^{p+k-1}  \\[3pt] 
1 & x_2 & x_2^2 & \cdots & x_2^{k-2} & x_2^{p+k-1}  \\[3pt] 
\vdots & \vdots & \vdots & \ddots & \vdots & \vdots \\[3pt]
1 & x_k & x_k^2 & \cdots & x_k^{k-2} & x_k^{p+k-1}  \\  
\end{bmatrix}.
\end{align*}
We reindexed the final sum to reflect the fact that the matrices in the second-to-last line
have repeated columns for $i=0,1,\ldots,k-2$ and hence have vanishing determinant.
To establish \eqref{eq:Mo}, and hence Theorem \ref{Theorem:Combo}, it suffices to show that
\begin{equation}\label{eq:Vhp}
\det\small
\begin{bmatrix} 
1 & x_1 & x_1^2 & \cdots & x_1^{k-2} & x_1^{p+k-1}  \\[3pt] 
1 & x_2 & x_2^2 & \cdots & x_2^{k-2} & x_2^{p+k-1}  \\[3pt] 
\vdots & \vdots & \vdots & \ddots & \vdots & \vdots \\[3pt]
1 & x_k & x_k^2 & \cdots & x_k^{k-2} & x_k^{p+k-1}  \\  
\end{bmatrix}\normalsize
= h_p(x_1,\ldots,x_k) \det V(x_1,\ldots,x_k) .
\end{equation}

\subsection{Some algebraic combinatorics}\label{Subsection:AlgCom}
To establish \eqref{eq:Vhp} requires a small amount of algebraic combinatorics.  We briefly review the notation and results necessary for this purpose; the interested reader may consult \cite{Stanley2} for complete details.

Let $\lambda :=(\lambda_1,\lambda_2,\ldots,\lambda_k)$ denote the integer partition
\begin{equation*}
p = \lambda_1 + \lambda_2 + \cdots + \lambda_k,
\end{equation*}
in which  $\lambda_1 \geq \lambda_2 \geq \cdots \geq \lambda_k \geq 0$.  To such a partition we 
associate the polynomial
\begin{equation*}
 a_{(\lambda_1+k-1, \lambda_2+k-2, \ldots , \lambda_k)} (x_1, x_2, \ldots , x_k) 
 :=\det  \small
 \begin{bmatrix} 
 x_1^{\lambda_1+k-1} & x_2^{\lambda_1+k-1} & \ldots & x_k^{\lambda_1+k-1} \\[3pt]
 x_1^{\lambda_2+k-2} & x_2^{\lambda_2+k-2} & \ldots & x_k^{\lambda_2+k-2} \\[3pt]
 \vdots & \vdots & \ddots & \vdots \\[3pt]
 x_1^{\lambda_k} & x_2^{\lambda_k} & \ldots & x_k^{\lambda_k} 
 \end{bmatrix} ,
\end{equation*}
which is an alternating function of the variables $x_1,x_2,\ldots,x_k$
(interchanging any two of the variables changes the sign of the determinant).
As an alternating polynomial, the preceding 
is divisible by
\begin{align*}
 a_{(k-1, k-2, \ldots , 0)} (x_1, x_2, \ldots , x_k) 
& = \det \small
 \begin{bmatrix} 
 x_1^{k-1} & x_2^{k-1} & \ldots & x_n^{k-1} \\[3pt]
 x_1^{k-2} & x_2^{k-2} & \ldots & x_n^{k-2} \\[3pt]
 \vdots & \vdots & \ddots & \vdots \\[3pt]
 1 & 1 & \ldots & 1 
 \end{bmatrix}  
 \\
& = (-1)^{\binom{n}{2}} \det V(x_1,x_2,\ldots,x_k).
\end{align*}

The \emph{Schur polynomial} in the variables $x_1,x_2,\ldots,x_k$ corresponding to the partition $\lambda$ is
\begin{equation*}
 s_{\lambda} (x_1, x_2, \ldots , x_k) 
 :=\frac{ a_{(\lambda_1+k-1, \lambda_2+k-2, \ldots , \lambda_k+0)} (x_1, x_2, \ldots , x_k)}{a_{(k-1,k-2, \ldots , 0)} (x_1, x_2, \ldots , x_k) };
\end{equation*}
this is Jacobi's bialternant identity (which is itself a special case of the famed
Weyl character formula). 
We now prove \eqref{eq:Vhp}.  If we consider the partition
\begin{equation*}
\lambda = (p,\underbrace{0,0,\ldots,0}_{\text{$k-1$ zeros}}).
\end{equation*}
then it is well known that $s_{\lambda}(x_1,x_2,\ldots,x_k) = h_p(x_1,x_2,\ldots,x_k)$,
and hence
\begin{align*}
&\det\small
\begin{bmatrix} 
1 & x_1 & x_1^2 & \cdots & x_1^{k-2} & x_1^{p+k-1}  \\[3pt] 
1 & x_2 & x_2^2 & \cdots & x_2^{k-2} & x_2^{p+k-1}  \\[3pt] 
\vdots & \vdots & \vdots & \ddots & \vdots & \vdots \\[3pt]
1 & x_k & x_k^2 & \cdots & x_k^{k-2} & x_k^{p+k-1}  \\  
\end{bmatrix}
\\
&\qquad= (-1)^{\binom{n}{2}} \det \small
\begin{bmatrix} 
 x_1^{p+k-1} & x_2^{p+k-1} & \ldots & x_k^{p+k-1} \\[3pt]
 x_1^{k-2} & x_2^{k-2} & \ldots & x_k^{k-2} \\[3pt]
 \vdots & \vdots & \ddots & \vdots \\[3pt]
 1 & 1 & \ldots & 1 
 \end{bmatrix} \\
&\qquad =(-1)^{\binom{n}{2}}  a_{(p+k-1, k-2, \ldots , 0)} (x_1, x_2, \ldots , x_k) \\
&\qquad =(-1)^{\binom{n}{2}}  s_{\lambda}(x_1,x_2,\ldots,x_k) 
a_{(k-1,k-2,\ldots,0)}(x_1,x_2,\ldots,x_k) \\
&\qquad=  h_p(x_1,x_2,\ldots,x_k) \det V(x_1,x_2,\ldots,x_k)  .
\end{align*}
This establishes \eqref{eq:Vhp} and concludes the proof of Theorem \ref{Theorem:Combo}.\qed

%%%%%%%%%%%%%%%%%%%%%%%%%%%%
\section{Proof of Theorem \ref{Theorem:Main}}\label{Section:ProofMain}
The proof of Theorem~\ref{Theorem:Main} uses a few tools, such as weak convergence and Fourier transforms of measures, 
that are not standard in the study of numerical semigroups.
Since these ideas are not required to apply Theorem~\ref{Theorem:Main} and are not used elsewhere in the paper, we introduce the
required concepts as needed and make no attempt to 
state definitions and lemmas in the greatest possible generality.  

We begin with the necessary background on moments of probability measures 
(Subsection~\ref{Subsection:MM}), Fourier transforms of measures (Subsection~\ref{Subsection:Fourier}),
and characteristic functions (Subsection~\ref{Subsection:Characteristic}).
We then prove a power series convergence lemma (Subsection~\ref{Subsection:Power})
to set up our use of characteristic functions.  We introduce a family of singular measures
(Subsection~\ref{Subsection:Nun}) that converge weakly to the desired probability measure.
This is established using the method of characteristic functions and L\'evy's continuity theorem
(Subsection~\ref{Subsection:Phi}).  We wrap things up with a dose of Fourier inversion and some
detailed computations (Subsection~\ref{Subsection:Completing}).  
%%%%%%%%%%%%%%%%%%

\subsection{Measures and moments}\label{Subsection:MM}
A \emph{Borel measure} is a measure defined on the Borel $\sigma$-algebra, the $\sigma$-algebra of subsets of $\RR$
generated by the open sets.  Every subset of $\RR$ we consider in this paper is a Borel set.
Let $\nu$ be a \emph{probability measure} on $[0,1]$; that is, $\nu$ is a Borel
measure on $[0,1]$
such that $\nu([0,1]) =1$ and $\nu(A) \geq 0$ for every Borel set $A \subseteq [0,1]$.
For $p \in \NN$, the $p$th \emph{moment} of $\nu$ is
\begin{equation*}
m_p(\nu) = \int_0^1 t^p \,d\nu(t).
\end{equation*}
The moments of $\nu$ are uniformly bounded since
\begin{equation}\label{eq:UniformBound}
0\leq m_p(\nu) \leq \int_0^1d\nu(t) = \nu\big([0,1]\big) = 1.
\end{equation}
A probability measure on $[0,1]$ is completely determined by its moments \cite[Thm.~30.1]{Billingsley}.

Let $\nu_n$ be a sequence of probability measures on $[0,1]$.
Then $\nu_n$ \emph{converges weakly} to a measure $\nu$ on $[0,1]$, denoted by
$\nu_n \to \nu$, if any of the following equivalent conditions hold \cite[Thm.~25.8, Thm.~30.2]{Billingsley}:
\begin{enumerate}
\item $\displaystyle\lim_{n\to\infty} \int_0^1 f(t)\,d\nu_n(t) = \int_0^1 f(t)\,d\nu(t)$ 
for every continuous function on $[0,1]$. 
\item $m_p(\nu_n)\to m_p(\nu)$ for all $p\in \NN$.
\item $\nu_n(A) \to \nu(A)$ for every Borel set $A \subseteq [0,1]$ 
for which $\nu(\partial A) = 0$; that is, $\nu$ places no mass on the boundary of $A$.
\end{enumerate}
The equivalence of (a) and (b) follows from the Weierstrass approximation theorem: every continuous
function on $[0,1]$ is uniformly approximable by polynomials.
The weak limit of a sequence of probability measures is a probability measure and the limit measure is unique \cite[p.~336-7]{Billingsley}.

%%%%%%%%%%%%%%%%%%%%%%%%
\subsection{The Fourier transform}\label{Subsection:Fourier}
The \emph{Fourier transform} of a probability measure $\nu$ on $[0,1]$ is
\begin{equation}\label{eq:Alternate}
\widehat\nu(z) := \int_0^1 e^{i t z}\,d\nu(t).
\end{equation}
This may differ in appearance from what the reader is accustomed to.
Normally one integrates over $\RR$ in \eqref{eq:Alternate}, but that is unnecessary here because 
$\nu$ is supported on $[0,1]$.  We adhere to the positive sign in the exponent of the integrand in \eqref{eq:Alternate},
which is standard in probability theory \cite[Sect.~26]{Billingsley}.  Consequently,
the reader should be aware of potential sign discrepancies between what follows
and formulas from their favored sources.

The \emph{inverse Fourier transform} of a suitable $f:\RR\to\CC$ is
\begin{equation*}
\widecheck{f}(x) := \frac{1}{2\pi}\int_{\RR}f(t)e^{-i x t}\,dt.
\end{equation*}
With much additional work, the inverse Fourier transform can be defined
on distributions (``generalized functions'').
A friendly introduction to Fourier transforms of distributions is \cite[Ch.~4]{Osgood}.
In the sense of distributions, one can show
\begin{equation}\label{eq:FourierMain}\qquad\qquad
\widecheck{\bigg( \frac{e^{ia t}}{t^n }\bigg)}(x)
= \frac{|x-a| (x-a)^{n-2} }{2i^n(n-1)!},\qquad \text{for $n\geq 2$}.
\end{equation}
This follows from \cite[Ex.~9, p.~340]{Folland} or \cite[Table A-6]{Kammler}
and the standard translation identity \cite[eq.~(9.29)]{Folland}.  The method of finite parts for treating
highly singular functions as distributions is discussed in \cite[p.~324-5]{Folland}.

%%%%%%%%%%%%%%%%%%%%%%%%%%%%%
\subsection{Characteristic functions}\label{Subsection:Characteristic}
The \emph{characteristic function} $\phi_{\nu}$ of a probability measure $\nu$ on $[0,1]$ 
is the Fourier transform of $\nu$:
\begin{align*}
\phi_{\nu}(z) 
&:= \widehat{\nu}(z)
= \int_0^1 e^{itz}\,d\nu(t) = \int_0^1 \sum_{p=0}^{\infty}\frac{(itz)^p}{p!}\,d\nu(t)\\
&= \sum_{p=0}^{\infty} \frac{(iz)^p}{p!} \int_0^1 t^p\,d\nu(t) 
= \sum_{p=0}^{\infty} m_p(\nu) \frac{(iz)^p}{p!}.
\end{align*}
The interchange of sum and integral is permissible because for each fixed $z\in \CC$ 
the series involved converges uniformly for $t\in [0,1]$.  Since $|m_p(\nu)| \leq 1$ for all 
$p \in \NN$, comparison with the exponential series ensures that the series above 
has an infinite radius of convergence and hence $\phi_{\nu}$ is an entire function.  Moreover,
\begin{equation}\label{eq:Bounded0}
|\phi_{\nu}(x)| = \bigg|\int_0^1 e^{itx}\,d\nu(t)\bigg| \leq \int_0^1d\nu(t) = 1
\end{equation}
for all $x \in \RR$ since $\nu$ is a probability measure on $[0,1]$.
If $\nu_1$ and $\nu_2$ are probability measures and 
$\phi_{\nu_1} = \phi_{\nu_2}$, then $\nu_1=\nu_2$ \cite[Thm.~26.2]{Billingsley}.

Under certain circumstances, we can recover a probability measure from its characteristic function \cite[p.347-8]{Billingsley}.

\begin{Lemma}[Inversion Theorem]\label{Lemma:Inversion}
If $\int_{\RR}|\phi_{\nu}(t)|\,dt$ is finite, then 
$F_{\nu} := \widecheck{\phi_{\nu}}$ is a bounded continuous function 
and $\nu(A) = \int_A F_{\nu}(t)\,dt$ for every Borel set $A$.
\end{Lemma}

The following theorem  of L\'evy relates weak limits of probability
measures to pointwise convergence of the corresponding characteristic functions \cite[Thm.~26.3]{Billingsley}.

\begin{Lemma}[L\'evy's Continuity Theorem]\label{Lemma:Levy}
Let $\nu_n$ be probability measures on $[0,1]$ such that
$\phi_{\nu_n}$ converges pointwise on $\RR$ to a continuous function 
$\phi$.  Then 
\begin{enumerate}
\item $\phi = \phi_{\nu}$ for some probability measure $\nu$ on $[0,1]$;
\item $\nu_n \to \nu$ (weak convergence of measures);
\item $m_p(\nu) =  \lim_{p\to\infty} m_p(\nu_n)$ for all $p \in \NN$.
\end{enumerate}
\end{Lemma}

\subsection{A power series lemma}\label{Subsection:Power}
We ultimately plan to consider a sequence of probability measures $\nu_n$ 
to which we will apply Lemma \ref{Lemma:Levy}.  To show that the associated
sequence of characteristic functions converges we need the following lemma.

\begin{Lemma}\label{Lemma:PowerSeries}
Suppose that $|a_p(n)| \leq 1$ for all $n,p \in \NN$.
If $\lim_{n\to\infty}a_p(n) = a_p$ for each $p$, then 
$\sum_{p=0}^{\infty} a_p(n) \frac{z^p}{p!}$
converges locally uniformly on $\CC$ to $\sum_{p=0}^{\infty} a_p \frac{z^p}{p!}$
\end{Lemma}

\begin{proof}
Fix $R >0$.  Let $N\in \NN$ be so large that
\begin{equation*}
\sum_{p=N}^{\infty} \frac{R^p}{p!} < \frac{\epsilon}{4}.
\end{equation*}
Let $M \in \NN$ be such that
\begin{equation*}\qquad
n \geq M \quad \implies \quad |a_p(n) - a_p| \frac{R^p}{p!} < \frac{\epsilon}{2N}
\quad \text{for $0 \leq p \leq N-1$}.
\end{equation*}
If $|z| \leq R$, then
\begin{align*}
\left| \sum_{p=0}^{\infty} a_p(n) \frac{z^p}{p!} - \sum_{p=0}^{\infty} a_p \frac{z^p}{p!}\right|
&\leq \sum_{p=0}^{N-1} |a_p(n)-a_p| \frac{|z|^p}{p!}+ \sum_{p=N}^{\infty} |a_p(n)-a_p| \frac{|z|^p}{p!} \\
&\leq \sum_{p=0}^{N-1} |a_p(n)-a_p| \frac{R^p}{p!}+ 2 \sum_{p=N}^{\infty}  \frac{R^p}{p!} \\
&< \frac{\epsilon}{2} + \frac{\epsilon}{2} \\
&= \epsilon.
\end{align*}
Thus, the convergence is uniform on $|z|\leq R$.  Since $R>0$ was arbitrary, the 
convergence is locally uniform on $\CC$. 
\end{proof}

\subsection{The measures $\nu_n$}\label{Subsection:Nun}
Fix $S = \<n_1,n_2,\ldots,n_k\>$.
Consider the probability measures 
\begin{equation}\label{eq:nun}
\nu_n = \frac{1}{|\Le\multi{n}|}\sum_{\ell \in \Le\multi{n}} \delta_{\ell /n}
\end{equation}
for $n \in \NN$, wherein $\delta_x$ denotes the point mass at $x$.  Since
\begin{equation}\label{eq:MinMax}
 \frac{1}{n_k} \leq \underset{\ell \in \Le\multi{n}}{\operatorname{min}}\frac{ \ell}{n} 
\quad \text{and} \quad
\underset{\ell \in \Le\multi{n}}{\operatorname{max}}\frac{ \ell}{n} \leq \frac{1}{n_1},
\end{equation}
the support of each $\nu_n$ is contained in $[\frac{1}{n_k},\frac{1}{n_1}] \subset [0,1]$.
The $p$th moment of $\nu_n$ is
\begin{equation*}
m_p(\nu_n)
= \int_0^1 t^p \,d\nu_n(t)
= \frac{1}{\Le\multi{n}}\sum_{\ell \in \Le\multi{n}}\left(\frac{\ell}{n}\right)^p .
\end{equation*}
Theorem \ref{Theorem:Moment}, \eqref{eq:Znk}, and the preceding imply that
\begin{equation*}
\lim_{n\to\infty} m_p(\nu_n) = \binom{p+k-1}{p}^{-1} h_p\left( \frac{1}{n_1}, \frac{1}{n_2} ,\ldots, \frac{1}{n_k} \right).
\end{equation*}

\subsection{The function $\phi$}\label{Subsection:Phi}
For notational simplicity we let $x_i = 1/n_i$ for $i\in [k]$.
The bound \eqref{eq:UniformBound}
and Lemma \ref{Lemma:PowerSeries} ensure that 
the characteristic functions $\phi_{\nu_n}$ converge locally uniformly
(and hence pointwise) on $\CC$ to
\begin{align}
\phi(z) 
&=\sum_{p=0}^{\infty} \binom{p+k-1}{p}^{-1} h_p(x_1,x_2,\ldots,x_k)\frac{(iz)^p}{p!}  \label{eq:phiPower} \\
&=(k-1)!\sum_{p=0}^{\infty} \frac{h_p(x_1,x_2,\ldots,x_k)}{(p+k-1)!} (iz)^p \nonumber \\
&= (k-1)!\sum_{r=1}^k \frac{e^{ix_r z}}{(iz)^{k-1}\prod_{j\neq i}(x_r-x_j)} ,  \label{eq:phiExp}
\end{align}
in which the final equality is Theorem \ref{Theorem:Combo}.
A glance at \eqref{eq:phiPower}, or Theorem \ref{Theorem:Combo} itself,
tells us that the apparent singularity in \eqref{eq:phiExp} at $z=0$ is removable.
In particular, $\phi$ is an entire function and $\phi(0)=1$.

Lemma \ref{Lemma:Levy} (L\'evy's Continuity theorem) provides a probability 
measure $\nu$ such that $\nu_n\to\nu$ and $\phi = \phi_{\nu}$.
From \eqref{eq:Bounded0} we see that $\phi$ is bounded on $\RR$.  Moreover,
\begin{equation*}
\int_{1}^{\infty} \left| \frac{e^{ix_rt}}{t^{k-1}} \right|\,dt
+\int_{-\infty}^{-1} \left| \frac{e^{ix_rt}}{t^{k-1}} \right|\,dt 
\,\leq\, 2\int_1^{\infty}\frac{dt}{t^2} = 2
\end{equation*}
and hence \eqref{eq:phiExp} implies that 
$\int_{\RR} |\phi(t)|\,dt$ is finite.  Lemma \ref{Lemma:Inversion} implies that
$F_{\nu} = \widecheck{\phi_{\nu}}$ is a bounded continuous function such that
\begin{equation}\label{eq:nuAF}
\nu(A) = \int_A F_{\nu}(t)\,dt
\end{equation}
for all Borel sets $A \subseteq [0,1]$.  In what follows, we let $F:= F_{\nu}$.

\subsection{Completion of the proof}\label{Subsection:Completing}
We maintain the convention that $x_i = 1/n_i$ for $i\in [k]$.
From the preceding discussion we have
\begin{align}
F(x)
&= \widecheck{\phi_{\nu}}(x) \nonumber\\
&= \bigg((k-1)!\sum_{r=1}^k \frac{e^{ix_r z}}{(iz)^{k-1}\prod_{j\neq r}(x_r-x_j)}  \bigg)^{\!\!\widecheck{\quad}}\!\!\!\!(x) && \text{(by \eqref{eq:phiExp})} \label{eq:Backward}\\
&=\frac{(k-1)!}{i^{k-1}}\sum_{r=1}^k \frac{1}{\prod_{j\neq r}(x_r-x_j)} 
\widecheck{\bigg( \frac{e^{ix_r z}}{z^{k-1} }\bigg)}(x) \nonumber \\
&=\frac{(k-1)!}{2(i^{k-1})^2(k-2)!}
\sum_{r=1}^k \frac{|x-x_r|(x-x_r)^{k-3}}{\prod_{j\neq r}(x_r-x_j)}    && \text{(by \eqref{eq:FourierMain})}\nonumber\\
 &=  \frac{k-1}{2}\sum_{r=1}^k\frac{|x_r-x|(x_r-x)^{k-3}}{\prod_{j\neq r}(x_r-x_j)} \label{eq:MidT} \\
&= \frac{k-1}{2}\sum_{r=1}^k \frac{|\frac{1}{n_r}-x|(\frac{1}{n_r}-x)^{k-3}}{\prod_{j\neq r}(\frac{1}{n_r}-\frac{1}{n_j})}   \nonumber\\
&=\frac{(k-1)n_1n_2\cdots n_k}{2}\sum_{r=1}^k \frac{|1-n_rx|(1-n_r x)^{k-3}}{\prod_{j\neq r}(n_j-n_r)}  .
\label{eq:FinalT}
\end{align}

For $k\geq 3$, induction and the definition of the derivative confirms that $|x|x^{k-3}$ is 
$k-3$ times continuously differentiable on $\RR$, but is not differentiable $k-2$ times at $x=0$.  
Since the $n_1,n_2,\ldots,n_k$ are distinct
and because the zeros of $1-n_r x$ belong to $[0,1]$, we conclude that $F$
is $k-3$ times continuously differentiable on $[0,1]$, but not differentiable $k-2$ times there.

Let $[\alpha,\beta] \subseteq [0,1]$.  Observe that $\partial[\alpha,\beta]=\{\alpha,\beta\}$, so \eqref{eq:nuAF} implies
\begin{equation*}
\nu(\{\alpha,\beta\}) = \nu(\{\alpha\}) + \nu(\{\beta\})
= \int_{\alpha}^{\alpha} F(t)\,dt + \int_{\beta}^{\beta} F(t)\,dt = 0.
\end{equation*}
Characterization (c) of
the weak convergence $\nu_n \to \nu$ (Subsection \ref{Subsection:MM}) implies
\begin{align*}
\lim_{n\to\infty} \frac{ |\{ \ell : \Le\multi{n} : \ell \in[\alpha n, \beta n]\} |}{ \Le\multi{n}}
&= \lim_{n\to\infty} \frac{ |\{ \ell : \Le\multi{n} : \frac{\ell}{n} \in[\alpha , \beta ]\}| }{ \Le\multi{n}} \\
&= \lim_{n\to\infty} \nu_n\big( [\alpha,\beta]\big) && \text{by \eqref{eq:nun}}\\
&= \nu\big([\alpha,\beta]\big) \\
&= \int_{\alpha}^{\beta} F(t)\,dt.
\end{align*}

Now observe that $F$ is supported on $[1/n_k,1/n_1]$ since \eqref{eq:MinMax} ensures that
\begin{equation*}
\nu([a,b]) = \lim_{n\to\infty} \nu_n([a,b]) = 0
\end{equation*}
for any interval $[a,b]$ that does not intersect $[1/n_k,1/n_1]$.
Consequently, characterization (a) of
the weak convergence $\nu_n \to \nu$ and Lemma \ref{Lemma:Inversion} yield
\begin{align*}
\lim_{n\to\infty}  \frac{1}{ \Le\multi{n}} \sum_{\ell \in \Le\multi{n}} g(\ell/n) 
&= \lim_{n\to\infty} \int_0^1 g(t) \,d\nu_n(t) \\
&= \int_0^1 g(t) \,d\nu(t) \\
&= \int_0^1 g(t) F(t)\,dt
\end{align*}
for any continuous function $g:(0,1)\to\CC$ (since $F$ is supported on $[1/n_k,1/n_1] \subset (0,1)$, the values of $g$
outside of this interval are irrelevant).  This concludes the proof of Theorem \ref{Theorem:Main}. \qed

%%%%%%%%%%%%%%%%%%%%%
\section{Concluding remarks}\label{Section:End}
Theorem \ref{Theorem:Main} (which depends upon Theorems \ref{Theorem:Moment} and \ref{Theorem:Combo})
appears to answer all questions about the asymptotic behavior of factorization length multisets in numerical semigroups.
However, there are a few issues it does not immediately address which suggest several avenues for further exploration.

Although the length distribution function $F$ provided by Theorem \ref{Theorem:Main} is explicit,
it is no longer amenable to symbolic computation when a semigroup has a large number of generators.  
That is, one typically does not expect closed-form answers in terms of $n_1,n_2,\ldots,n_k$.

We have shown analytically 
that $F$ is unimodal for $k=3,4$.   Numerical evidence suggests this remains true for $k \geq 5$,
although we do not have a general proof.  This would lead to a better understanding of the asymptotic relationship between the generators and the mode of the factorization lengths.  Finally, the numerical examples in Section \ref{Section:Applications} suggest relatively rapid convergence in parts~(a) and ~(c) of Theorem \ref{Theorem:Main}.
It would be of some interest to prove this analytically.  However, the techniques involved in the proof of these theorems do not
appear to readily admit quantitative estimates.

\bibliography{FLD4AS2-Probability}
\bibliographystyle{amsplain}

\end{document}